\documentclass[reqno]{amsart} 
 
\usepackage{hyperref}
\usepackage{amsmath}
\usepackage{amssymb}
\usepackage{amsfonts}
\usepackage{amsthm}
\usepackage{latexsym}
\usepackage{esint}
\usepackage{mathtools}
\usepackage{mathrsfs}
\usepackage{graphicx}
\usepackage{wrapfig}
\usepackage{bbm}
\usepackage{color}
\usepackage{bm}
\usepackage[top=1in, bottom=1.25in, left=1.10in, right=1.10in]{geometry}
\usepackage{todonotes}
\allowdisplaybreaks

\newtheorem{theorem}{Theorem}[section]

\newtheorem{proposition}[theorem]{Proposition}
\newtheorem{corollary}[theorem]{Corollary}

\theoremstyle{definition}
\newtheorem{definition}[theorem]{Definition}

\theoremstyle{remark}
\newtheorem{remark}[theorem]{Remark}

\numberwithin{equation}{section}

\newcommand{\bx}{\mathbf{x}}
\newcommand{\by}{{y}}

\newcommand{\bq}{\mathbf{q}}
\newcommand{\bu}{\mathbf{u}}
\newcommand{\bw}{\mathbf{w}}
\newcommand{\bfu}{\mathbf{u}}

\newcommand{\bff}{\mathbf{f}}

\newcommand{\bv}{\mathbf{v}}
\newcommand{\bn}{\mathbf{n}}

\newcommand{\dy}{\, \mathrm{d}y}
\newcommand{\dd}{\,\mathrm{d}}
\newcommand{\dq}{\, \mathrm{d} \mathbf{q}}
\newcommand{\dx}{\, \mathrm{d} \mathbf{x}}

\newcommand{\dt}{\, \mathrm{d}t}
\newcommand{\ds}{\, \mathrm{d}s}

\newcommand{\divx}{\mathrm{div}_{\mathbf{x}}}
\newcommand{\divq}{\mathrm{div}_{\mathbf{q}}}

\newcommand{\delx}{\Delta_{\mathbf{x}}}

\newcommand{\nabx}{\nabla_{\mathbf{x}}}
\newcommand{\naby}{\partial_y}
\newcommand{\nabq}{\nabla_{\mathbf{q}}}

\newcommand{\Delx}{\Delta_{\mathbf{x}}}

\newcommand{\bT}{\mathbb{T}}
\newcommand{\bU}{\mathbb{U}}

\newcommand{\R}{\mathbb{R}}

\newcommand{\Oeta}{\Omega_{\eta}}
\newcommand{\Ozeta}{\Omega_{\zeta}}

\begin{document} 

\title[The inviscid limit]{Equilibration and convected limit in 2D-1D corotational Oldroyd's fluid-structure interaction}  


%
%
\author{Prince Romeo Mensah}
\address{Faculty of Mathematics, University of Duisburg-Essen, Thea-Leymann-Strasse 9, 45127 Essen,
Germany}


\date{\today}



\begin{abstract} 
We consider a solute-solvent-structure mutually coupled system of equations given by an Oldroyd-type model for a  two-dimensional dilute corotational polymer fluid with solute diffusion and damping that is interacting with a one-dimensional viscoelastic shell. Firstly, we give the rate at which its solution decays exponentially in time to the equilibrium solution, independent of the choice of the initial datum. Secondly, as the polymer relaxation time goes to infinity (or, equivalently, the center-of mass diffusion goes to zero), we show that any family of strong solutions of the system described above, that is parametrized by the 
relaxation time, converges to an essentially bounded weak solution of a  corotational polymer fluid-structure interaction system whose solute evolves according to the convected time derivative of its extra stress tensor. A consequence of this is a weak-strong uniqueness result. 

%
\end{abstract}

\maketitle

\section{Introduction} 
As the name suggest, incompressible fluids are those fluids whose volume  cannot  contract or expand, and thus, are volume preserving.  Therefore,  to avoid cracks  or breakages, it is very important that any flexible receptacle containing such a fluid does not shrink in volume. 
Interpreted analytically,  the transformation of \textit{a region containing an incompressible fluids} (not the incompressible fluids itself) should be at least volume preserving in order to avoid a degeneracy.  
Now consider a well-defined injective mapping $\bm{\Psi}_\eta$ that maps a reference domain $\Omega$ to a current state $\Omega_\eta$.
Since the volume  of the current state $\Omega_\eta$ is given by
\begin{align*}
\vert\Omega_\eta\vert=\int_\Omega 
\vert\det(\nabx \bm{\Psi}_\eta)\vert \dx,
\end{align*}
we find that $\vert\Omega_\eta\vert=1$ if  
\begin{align*}
\vert\det(\nabx \bm{\Psi}_\eta)\vert =\frac{1}{\vert\Omega_\eta\vert}.
\end{align*}
This tells us that the transformation $ \bm{\Psi}_\eta$ shrinks or scale down each infinitesimal volume in $\Omega$ by a factor $\tfrac{1}{\vert\Omega_\eta\vert}$ in order to achieve the unit volume  $\vert\Omega_\eta\vert=1$. As explained earlier, such a shrinkage is inimical in the context of incompressible fluid-structure interaction. Thus, within the physical context of 
\begin{align*}
\vert\det(\nabx \bm{\Psi}_\eta)\vert \neq \frac{1}{\vert\Omega_\eta\vert},
\end{align*}
we consider a dilute corotational incompressible polymeric fluids of Oldroyd type. Such a fluid is a solvent-solute mixture whose solvent subcomponent has velocity field $\mathbf{u}:(t, \mathbf{x})\in I \times \Oeta \mapsto  \mathbf{u}(t, \mathbf{x}) \in \mathbb{R}^2$ defined on the spatially flexible spacetime cylinder
\begin{align*}
I \times \Oeta:=\bigcup_{t\in I}\{t\}\times \Omega_{\eta(t)}
\end{align*}
and satisfy the incompressibility condition
\begin{align*}
\divx\bu=0 
\end{align*}
in $I\times\Oeta$. 
The solute subcomponent is a `dumbbell' consisting of two massless monomers connected by a Hookean spring. The elongation vector $\bq\in\mathbb{R}^2$ between the position of these two monomers has size $|\bq|>0$ and the spring potential and Maxwellian are, respectively,
\begin{align*} 
U\Big(\frac{1}{2}|\bq|^2 \Big)=\frac{1}{2}|\bq|^2, \qquad\qquad M=\frac{\exp(- \tfrac{1}{2}|\bq|^2)}{\int_B\exp(- \tfrac{1}{2}|\bq|^2)\dq}.
\end{align*} 
If $\ell_0\ll\mathrm{diam}(\Oeta)$ is the characteristic microscopic length scale  of the characteristic dumbbell size, we can introduce the center-of-mass diffusion coefficient $\varepsilon\geq0$ of the dumbbell and the relaxation time $\lambda>0$ which characterizes the elastic relaxation property of the fluid. More precisely, these two parameters are given by
\begin{align}
\label{scales}
\varepsilon=\frac{\ell^2_0}{8\lambda}, \qquad \lambda=\frac{\zeta}{4H}
\end{align}  
where $\zeta$ is the friction coefficient and $H$ is the spring constant. With this, the mesoscopic description of the solute subcomponent is given by the probability density function  $$f :(t, \mathbf{x} ,\bq)\in I \times \Oeta\times \mathbb{R}^2  \mapsto f(t, \mathbf{x} ,\bq) \in \mathbb{R}_{\geq0}$$ which evolves according to the Fokker--Planck equation
\begin{align}
\label{fokkerPlanck}
\partial_t f+ (\mathbf{u}\cdot \nabx)f+\divq(\mathbb{W}(\nabx\bu)\bq f)=\varepsilon\Delx f+\frac{1}{2\lambda}\divq\left(M\nabq\left(\frac{f}{M}\right)\right)
\end{align}
in $I\times\Oeta\times \mathbb{R}^2$. Here,
\begin{align*}
\mathbb{W}(\nabx \bu)=\tfrac{1}{2}(\nabx \bu-(\nabx \bu)^\top)
\end{align*}
is the vorticity tensor of the solvent or the antisymmetric gradient of the solvent's velocity.
 To obtain a macroscopic description for the solute subcomponent from the mesoscopic description above, we consider a tensor-valued average and a scalar average given by the stress tensor $\widehat{\bT}:(t, \mathbf{x} )\in I \times \Oeta  \mapsto \widehat{\bT}(t, \mathbf{x} )\in\mathbb{R}^{2\times2}$ of the solute  and the polymer number density $\rho :(t, \mathbf{x} )\in I \times \Oeta  \mapsto \rho(t, \mathbf{x} ) \in \mathbb{R}$ which we assume to be a constant since only two monomers are at play. These averages are precisely given by
\begin{align*} 
\widehat{\bT}(t, \bx)= \int_{B} f(t,\bx,\bq)\bq\otimes\bq \dq,
\qquad
\rho(t, \bx)= \int_{B} f(t,\bx,\bq) \dq\equiv \mathrm{constant},
\end{align*}
respectively. If we now set 
\begin{align*}
\bT=\widehat{\bT}-k\rho\mathbb{I} 
\end{align*}  
where $k>0$ is the Boltzmann constant, we finally obtain the macroscopic description of the solute subcomponent
\begin{align}
\label{macroClosure1}
\partial_t \bT + (\mathbf{u}\cdot \nabx) \bT
=
\mathbb{W}(\nabx \bu)\bT + \bT\mathbb{W}((\nabx \bu)^\top) -\frac{1}{2\lambda}\bT+ \varepsilon\Delx  \bT 
\end{align}
in $I\times\Oeta $.

Further assumptions on the two parameters $\varepsilon\geq0$ and $\lambda>0$ leads to several variations of  \eqref{macroClosure1}. One such variation that is strongly contested in the literature is whether the center-mass-of diffusion parameter $\varepsilon$ should be zero or strictly positive. Indeed, as  explained in \cite{barrett2007existence}, when $L\approx 1$ is a characteristic macroscopic length scale, the authors in \cite{bhave1991kinetic} estimate the ratio
$\ell_0^2/L^2$ to be in the range of about $10^{-9}$ to $10^{-7}$. Although not zero, this is a very small number leading one school of thought to justify setting $\varepsilon=0$ in \eqref{fokkerPlanck} (and thus, also in \eqref{macroClosure1}). This gives the equation a damped hyperbolic structure and thus, present a relatively difficult equation to analyse mathematically. On the other hand, it is also perfectly justifiable to let $\varepsilon>0$ since the aforementioned ratio is not exactly zero. In this case, a parabolic structure is obtained and the presence of this diffusion term regularises the equation. 

One can also impose extra assumption on the second parameter $\lambda>0$. For example, if the friction coefficient $\zeta$ is several order of magnitude larger than the other parameters, the elastic relaxation parameter will also be very large and as such, one may assume that
$\tfrac{1}{2\lambda}\approx0$
leading to a purely elastic solute with infinite memory.
 
The last two terms in \eqref{macroClosure1} being zero gives the equation a purely hyperbolic structure with absolutely no form of dissipation. The resulting equation can be traced back to Oldroyd's seminal paper \cite{oldroyd1950formulation} in which he argue that to produce a frame-invariant equation, one needed to
evaluate the stress tensor   in such a way that it accounts for translation, rotation and even deformation. This lead to the convected time-derivative or what is now referred to as the Oldroyd derivative which conserves all moments just as the usual material (advected) derivative does. In particular, if we let $ \bU:(t, \mathbf{x} )\in I \times \Ozeta  \mapsto  \bU(t, \mathbf{x} )\in\mathbb{R}^{2\times2}$ be stress tensor 
for a dilute corotational elastic bead and spring dumbbells in a solvent with a velocity field $\bv:(t, \mathbf{x})\in I \times \Ozeta \mapsto  \bv(t, \mathbf{x}) \in \mathbb{R}^2$, then $\bU$
%
%
is governed by 
\begin{align} 
\label{oldroydDerivative}
\partial_t \bU + (\bv\cdot \nabx) \bU
=
\mathbb{W}(\nabx \bv)\bU + \bU\mathbb{W}((\nabx \bv)^\top)  
\end{align}
in $I\times\Ozeta $. 
We refer to the review \cite{hinch2021oldroyd} for more insight on Oldroyd's work. 

One of our goals will be the rigorous justification of the vanishing limit of the last two terms on the right-hand side of \eqref{macroClosure1} leading to  a solution of \eqref{oldroydDerivative}. To properly motivate this work, however, we begin by rigorously setting up the system.
\subsection{Setup}
We consider a spatial reference domain  $\Omega \subset \mathbb{R}^2$ whose boundary $\partial\Omega$ may consist of a flexible part $\omega \subset \mathbb{R}$ and a rigid part $\Gamma \subset\mathbb{R}$. However, because the analysis at the rigid part is significantly simpler, we shall identify the whole of $\partial \Omega$ with $\omega$ and endow it with periodic boundary conditions. Let $I:=(0,T)$ represent a time interval for a given constant $T>0$. 
We represent the time-dependent  displacement of the structure by $\eta:\overline{I}\times\omega\rightarrow(-L,L)$ where $L>0$ is a fixed length of the tubular neighbourhood of $\partial\Omega$ given by
\begin{align*}
S_L:=\{\bx\in \mathbb{R}^2\,:\, \mathrm{dist}(\bx,\partial\Omega
)<L \}.
\end{align*}
Now, for some $k\in\mathbb{N}$ large enough, we assume that $\partial\Omega$  is parametrized by an injective mapping $\bm{\varphi}\in C^k(\omega;\mathbb{R}^2)$ with $\naby \bm{\varphi}\neq0$ such that
\begin{align*}
\partial{\Omega_{\eta(t)}}=\big\{\bm{\varphi}_{\eta(t)}:=\bm{\varphi}(\by)+\bn(\by)\eta(t,\by)\, :\, t\in I, \by\in \omega\big\}.
\end{align*}
The set $\partial{\Omega_{\eta(t)}}$ represents the boundary of the flexible domain at any instant of time $t\in I$ and the vector $\bn(y)$ is a unit normal at the point $y\in \omega$. 
We also let $\bn_{\eta(t)}(y)$ be the corresponding normal of $\partial{\Omega_{\eta(t)}}$ at the spacetime point $y\in \omega$ and $t\in I$. Then for $L>0$ sufficiently small, $\bn_{\eta(t)}(y)$ is close to $\bn(y)$ and $\bm{\varphi}_{\eta(t)}$ is close to $\bm{\varphi}$. Since $\naby \bm{\varphi}\neq0$,  it will follow that
\begin{align*}
\naby \bm{\varphi}_{\eta(t)} \neq0 \quad\text{ and }\quad \bn(y)\cdot \bn_{\eta(t)}(y)\neq 0 
\end{align*}
for $y\in \omega$ and $t\in I$. Thus, in particular, there is no loss of strict positivity of the Jacobian determinant provided that $\Vert \eta\Vert_{L^\infty(I\times\omega)}<L$.

For the interior points, we  transform the  reference domain $\Omega$ into a time-dependent moving domain $\Omega_{\eta(t)}$  whose state at time $t\in\overline{I}$ is given by
\begin{align*}
\Omega_{\eta(t)}
 =\big\{
 \bm{\Psi}_{\eta(t)}(\bx):\, \bx \in \Omega 
  \big\}.
\end{align*}
Here,
\begin{align*}
\bm{\Psi}_{\eta(t)}(\bx)=
\begin{cases}
\bx+\bn(\by(\bx))\eta(t,\by(\bx))\phi(s(\bx))     & \quad \text{if } \mathrm{dist}(\bx,\partial\Omega)<L,\\
    \bx & \quad \text{elsewhere } 
  \end{cases}
\end{align*}
is the Hanzawa transform with inverse $\bm{\Psi}_{-\eta(t)}$ and where for a point $\bx$ in the neighbourhood of $\partial\Omega$, the vector $\bn(y(\bx))$ is the unit normal at the point $y(\bx)=\mathrm{arg min}_{y\in\omega}\vert\bx -\bm{\varphi}(y)\vert$. Also, $s(\bx)=(\bx-\bm{\varphi}(y(\bx)))\cdot\bn(y(\bx))$ and $\phi\in C^\infty(\mathbb{R})$ is a cut-off function that is $\phi\equiv0$ in the neighbourhood of $-L$ and $\phi\equiv1$ in the neighbourhood of $0$. Note that $\bm{\Psi}_{\eta(t)}(\bx)$ can be rewritten as
\begin{align*}
\bm{\Psi}_{\eta(t)}(\bx)=
\begin{cases}
\bm{\varphi}(y(\bx))+\bn(\by(\bx))[s(\bx)+\eta(t,\by(\bx))\phi(s(\bx)) ]    & \quad \text{if } \mathrm{dist}(\bx,\partial\Omega)<L,\\
    \bx & \quad \text{elsewhere. } 
  \end{cases}
\end{align*}
\\
With the above preparation in hand, we consider the corotational Oldroyd's model for the flow of a  dilute polymeric fluid interacting with a flexible structure.

For simplicity we set the length scale in \eqref{scales} to $\ell_0=2$ so that
\begin{align*}
\varepsilon=\frac{1}{2\lambda}.
\end{align*}
In so doing, $\lambda\rightarrow\infty$ if and only if $\varepsilon\rightarrow0$. We can thus perform a simultaneous limit of the diffusion and damping terms with an equivalent  choice of either a large relaxation time limit or a vanishing center-of-mass limit.
This simplification does not affect our work whatsoever and it fact, the same results holds true for any constant length scale $\ell_0>0$.

The unknown of our problem consists of a structure displacement function $\eta:(t, \by)\in I \times \omega \mapsto   \eta(t,\by)\in \mathbb{R}$, a fluid velocity field $\mathbf{u}:(t, \mathbf{x})\in I \times \Oeta \mapsto  \mathbf{u}(t, \mathbf{x}) \in \mathbb{R}^2$, a pressure function $p:(t, \mathbf{x})\in I \times \Oeta \mapsto  p(t, \mathbf{x}) \in \mathbb{R}$ and an extra stress tensor $\bT :(t, \mathbf{x} )\in I \times \Oeta  \mapsto \bT (t, \mathbf{x} ) \in \mathbb{R}^{2\times2}$
 such that the system of equations 
\begin{align}
\divx \bu=0, \label{divfree} 
\\
\partial_t \bu  + (\mathbf{u}\cdot \nabx)\mathbf{u} 
= 
 \nu\delx \bu -\nabx p 
+
\divx   \bT, \label{momEq}
\\
 \partial_t^2 \eta - \gamma\partial_t\partial_y^2 \eta +  \partial_y^4 \eta =  - ( \mathbb{S}\bn_\eta )\circ \bm{\varphi}_\eta\cdot\bn \,\det(\partial_y\bm{\varphi}_\eta), 
\label{shellEQ}
\\
\partial_t \bT + (\mathbf{u}\cdot \nabx) \bT
=
\mathbb{W}(\nabx \bu)\bT + \bT\mathbb{W}((\nabx \bu)^\top) - \varepsilon( 1- \Delx) \bT \label{solute}
\end{align}
holds on $I\times\Oeta\subset \mathbb R^{1+2}$ (with \eqref{shellEQ} posed on $I\times\omega\subset \mathbb R^{1+1}$) where
\begin{align*}
\mathbb{S}= \nu(\nabx \bu +(\nabx \bu)^\top) -p\mathbb{I}+  \bT,
\qquad
\mathbb{W}(\nabx \bu)=\tfrac{1}{2}(\nabx \bu-(\nabx \bu)^\top).
\end{align*}
The parameters $\nu$, $\gamma$ and $\varepsilon$ are positive constants, $\bn_\eta$ is the normal at $\partial\Oeta$ and $\mathbb{I}$ is the identity matrix.
We complement \eqref{divfree}--\eqref{solute} with the following initial and boundary conditions
\begin{align}
&\eta(0,\cdot)=\eta_0(\cdot), \qquad\partial_t\eta(0,\cdot)=\eta_\star(\cdot) & \text{in }\omega,
\\
&\bu(0,\cdot)=\bu_0(\cdot) & \text{in }\Omega_{\eta_0},
\\
&\bT(0,\cdot)=\bT_0(\cdot) &\text{in }\Omega_{\eta_0},
\label{initialCondSolv}
\\
&  
\bn_{\eta}\cdot\nabx\bT=0 &\text{on }I\times\partial\Omega_{\eta}.
\label{bddCondSolv}
\end{align}
Furthermore, we impose periodicity on the boundary of $\omega$ (with mean-zero elements in $\omega$) and the following interface condition
\begin{align} 
\label{interface}
&\bu\circ\bm{\varphi}_\eta=(\partial_t\eta)\bn & \text{on }I\times \omega
\end{align}
at the flexible part of the boundary with normal $\bn$.

The analysis of polymeric fluid-structure interaction problems was  initiated recently in \cite{breit2021incompressible} where the authors showed the existence of weak solutions to a dilute solute-solvent-structure mutually coupled system. This system consisted of the $3$-D noncorotational Fokker--Planck equation \eqref{fokkerPlanck} (where $\nabx\bu$ replaces $\mathbb{W}(\nabx\bu)$) for the mesoscopic description of the solute, the  $3$-D incompressible Navier--Stokes equation \eqref{divfree} and \eqref{momEq} giving the macroscopic description of the solvent, and with a  $2$-D structure modeled by a shell equation \eqref{shellEQ} of Koiter type (where the term $- \gamma\partial_t\partial_y^2 \eta +  \partial_y^4 \eta$ is replaced by the gradient of the so-called \textit{Koiter energy}). Uniqueness is unknown for this system but the solutions exist until potential degeneracies occur with the Koiter energy or with the structure deformation.
When the  $2$-D Koiter shell in \cite{breit2021incompressible} is replaced by the $2$-D viscoelastic shell equation \eqref{shellEQ}, the extension to the existence of a unique local-in-time strong solution was then shown in \cite{breit2023existence}. 
Note that for fixed spatial domains subject to periodic boundary conditions, one can construct solutions that are spatially more regular \cite{breit2021local}  than strong solutions. The corresponding result for the system with a structure displacement remains an interesting open problem even in lower dimensions.
 
In this paper, we analyse \eqref{divfree}--\eqref{interface} of Oldroyd type from two main point of views. To explain these, however, we first introduce the notion of a solution we are interested in.
%
%
%

\subsection{Concepts of solution and main results}
We begin this subsection with a precise definition of a 
 strong solutions of \eqref{divfree}--\eqref{interface}.
\begin{definition}[Strong solution]
\label{def:strongSolution}
Let $(\bT_0, \bu_0, \eta_0, \eta_\star)$
be a dataset that satisfies
\begin{equation}
\begin{aligned}
\label{mainDataForAllStrong}
&
\eta_0 \in W^{3,2}(\omega) \text{ with } \Vert \eta_0 \Vert_{L^\infty( \omega)} < L, \quad \eta_\star \in W^{1,2}(\omega),
\\&\bu_0 \in W^{1,2}_{\divx}(\Omega_{\eta_0} )\text{ is such that }\bu_0 \circ \bm{\varphi}_{\eta_0} =\eta_\star \bn \text{ on } \omega,
\\& 
\bT_0\in W^{1,2}(\Omega_{\eta_0}) 
\end{aligned}
\end{equation}
We call 
$(\eta, \bu, p, \bT)$ 
a \textit{strong solution} of   \eqref{divfree}--\eqref{interface} with dataset $(  \bT_0, \bu_0, \eta_0, \eta_\star)$  if:
\begin{itemize} 
\item the structure-function $\eta$ is such that $
\Vert \eta \Vert_{L^\infty(I \times \omega)} <L$ and
\begin{align*}
\eta \in &W^{1,\infty}\big(I;W^{1,2}(\omega)  \big)\cap L^{\infty}\big(I;W^{3,2}(\omega)  \big) \cap  W^{1,2}\big(I; W^{2,2}(\omega)  \big)
\\&\cap  W^{2,2}\big(I;L^{2}(\omega)  \big) \cap  L^{2}\big(I;W^{4,2}(\omega)  \big);
\end{align*}
\item the velocity $\bu$ is such that $\bu  \circ \bm{\varphi}_{\eta} =(\partial_t\eta)\bn$ on $I\times \omega$ and
\begin{align*} 
\bu\in  W^{1,2} \big(I; L^2_{\divx}(\Oeta ) \big)\cap L^2\big(I;W^{2,2}(\Oeta)  \big);
\end{align*}
\item the pressure $p$ is such that 
\begin{align*}
p\in L^2\big(I;W^{1,2}(\Oeta)  \big);
\end{align*}
\item the tensor $ \bT$ is such that 
\begin{align*}
\bT \in W^{1,2}\big(I;L^{2}(\Oeta)  \big) \cap L^\infty\big(I;W^{1,2}(\Oeta)  \big)\cap L^2\big(I;W^{2,2}(\Oeta)  \big);
\end{align*}
\item equations \eqref{divfree}--\eqref{solute} are satisfied a.e. in spacetime with $\eta(0)=\eta_0$ and $\partial_t\eta(0)=\eta_\star$ a.e. in $\omega$, as well as $\bfu(0)=\bfu_0$ and $\bT(0)=\bT_0$ a.e. in $\Omega_{\eta_0}$.
\end{itemize}
\end{definition} 
The existence of a unique strong solution of  \eqref{divfree}--\eqref{interface} in the sense of Definition \ref{def:strongSolution} has recently been shown in \cite{mensah2023weak} (see also \cite{mensah2024conditionally} for a conditional regularity result). With this solution in hand, our first goal is to study the rate at which the system \eqref{divfree}--\eqref{interface} attains equilibrium.
When working on domains without boundaries, one can use Fourier analysis to obtain pointwise bound for the Fourier transform of the velocity $\bu$. This bound then help control the viscous   term $\nu \Vert \nabx\bu\Vert_{L^2(\Omega)}$ leading to the equilibration estimate $ \Vert \bu(t)\Vert_{L^2(\Omega)}\leq c(1+t)^{-\alpha}$ for some $\alpha>0$. See for instant \cite{abdo2024long, kato1984strong, kajikiya19862, schonbek19852, schonbek1996decay, wiegner1987decay}.

When the solvent is driven by a trace-free velocity $\bu$ defined on a bounded domain $\Omega$, however, a direct application of the classical Poincar\'e inequality is used to control the viscous term $\nu \Vert \nabx\bu\Vert_{L^2(\Omega)}$ leading to the equilibration estimate $ \Vert \bu(t)\Vert_{L^2(\Omega)}\leq ce^{-c  t}$, see \cite[Theorem 5.1]{debiec2023corotational} and \cite{barrett2011existenceEquilibraI, barrett2012existenceEquilibraI}. 


A similar result in the context of fluid-structure interaction where the boundary of the domain evolves nontrivially  in time is unknown. This is our first task and the exact statement is the following:

\begin{theorem}
\label{thm:one}
Let $|\Oeta|\neq1$ and let $(\eta, \bu,   \bT)$ 
be a strong solution of   \eqref{divfree}--\eqref{interface}. Then the equilibration rates (until possible degeneracy)
\begin{align*}
&\Vert \bT(t)\Vert_{L^2(\Oeta)}\leq e^{- \varepsilon t}\Vert\bT_0\Vert_{L^2(\Omega_{\eta_0})},
\\&
\Vert \partial_t\eta(t)\Vert_{L^2(\omega)}\leq e^{- c_1\gamma t}\sqrt{
\Vert\eta_\star\Vert_{L^2(\omega)}^2+
\frac{1}{2\nu\varepsilon} (1-e^{-2\varepsilon t})
\Vert\bT_0\Vert_{L^2(\Omega_{\eta_0})}^2},
\\
&\Vert \bu(t) \Vert_{L^2(\Oeta)}
\leq 
e^{- \tfrac{c_2}{2}\nu\left(1-\left( \inf_{t\in I}\vert\Omega_\eta\vert\right)^{-1/2}\right)^2 t}
\sqrt{
\Vert \bu_0 \Vert_{L^2(\Omega_{\eta_0})}^2 
+
\frac{1}{2\nu\varepsilon}(1-e^{-2\varepsilon t})
\Vert  \bT_0 \Vert_{L^2(\Omega_{\eta_0})}^2 
}
\end{align*}
applies for any $t\in I$ where $c_1$ and $c_2$ are the Poincar\'e--Wirtinger constants on $\omega$ and $\Oeta$, respectively.
\end{theorem}
\begin{remark}
The difference in the rate for $\bu$ (i.e. the extra term $-|\Oeta|^{-1/2}$) compared to that established in \cite[Theorem 5.1]{debiec2023corotational} comes from the fact that the velocity field we consider does not have zero-trace. Indeed, since $\bu$ satisfies $\bu  \circ \bm{\varphi}_{\eta} =(\partial_t\eta)\bn$ on $\omega$, we obtain the extra rate for the shell velocity $\partial_t\eta$.
\end{remark}

\begin{remark}
It will become clear in the proof of Theorem \ref{thm:one} that these rates actually hold in the larger class of weak solutions, as defined in \cite{mensah2023weak}, by working on the Galerkin level and then passing to the limit to get the corresponding continuum estimates.
\end{remark}

In the mathematical analysis of polymeric fluids,  there are strong opinions on whether or not the center-of-mass diffusion terms (the terms with the $\varepsilon$ parameters) in \eqref{divfree}--\eqref{solute} should be added. A school of thought \cite{barrett2017existenceOldroyd, degond2009kinetic, schieber2006generalized} gives justifications for its inclusion while some others \cite{guillope1990existence, lin2005on} ignore it. The rigorous derivation (see for example \cite{suli2018mckean}) from the microscopic scale leads to the center-of-mass diffusion terms in the mesoscopic description \eqref{fokkerPlanck} and thus persist on the macroscopic level \eqref{solute}. However, the importance of this stress diffusion is also known to decrease as the length scale of the problem increases \cite{bhave1991kinetic}.  
Determining whether the solutions of a non-diffusive system can be obtained as limits of those from a corresponding diffusive system is a question of significant mathematical and practical relevance, see for instant  \cite[Page 486]{masmoudi2013global}
and \cite[Page 336]{debiec2023corotational}. Unfortunately, results of this kind are incredibly rare in the literature. That notwithstanding,  numerical experiments \cite{chupin2015stationary} suggests that  such a convergence hold at order $1$ in the $L^2_\bx$- and $W^{1,2}_\bx$-norms.

From earlier works \cite{breit2021incompressible, breit2023existence} on the Navier--Stokes--Fokker--Planck equation, a suggestion arose that a strong solution (which typically exists locally in time, cf. \cite{breit2023existence}) may converge to a weak solution (which will ideally exist globally in time barring any degeneracies, cf. \cite{breit2021incompressible}) as the center-of-mass diffusion parameter goes to zero. We, however, note that the notion of a strong solution, as constructed in \cite{breit2023existence}, only applies in spacetime and that the solution is interpreted weakly in the third independent variable, i.e. concerning the conformation vector. On the other hand, the strong solution of  \eqref{divfree}--\eqref{interface} constructed in \cite{mensah2023weak} does not have this mixed character since it is fully macroscopic with just a spacetime dependency. With the above discussion in hand,
we recently showed in \cite{mensah2025vanishing} that the system \eqref{divfree}--\eqref{solute}  with solute diffusion and damping  converges in $L^2_\bx$-norm to a corresponding system \textit{without} diffusion but \textit{with} damping in its solute subcomponent.   This confirmed the earlier numerical results \cite{chupin2015stationary} that suggested that  the diffusive model converges to the solution of the non-diffusive model at order 1 in the $L^2_\bx$- and $W^{1,2}_\bx$-norms. We have since learnt of another recent independent work \cite{wang2024} on the upper half plane that studies a similar problem to \cite{mensah2025vanishing}. In \cite{wang2024}, however, the full velocity gradient in the solute is considered but without a structure equation and the convergence hold in the $L^\infty_\bx$-norm.

Our second goal in this work is to show that the system \eqref{divfree}--\eqref{solute} which models a polymeric fluid of Oldroyd type with solute diffusion and damping that is interacting with a viscoelastic shell converges to the corotational polymeric fluid model of Oldroyd type \textit{without any form of diffusion and damping} that is also interacting with a viscoelastic shell. The latter model is described by the following system of equations
\begin{align}
\divx \bv=0, \label{divfreeL} 
\\
\partial_t \bv  + (\bv\cdot \nabx)\bv
= 
  \nu\delx \bv -\nabx \pi 
+ 
\divx   \bU, \label{momEqL}
\\
 \partial_t^2 \zeta - \gamma\partial_t\partial_y^2 \zeta +\partial_y^4 \zeta =  - ( \mathbb{S}_*\bn_\zeta )\circ \bm{\varphi}_\zeta\cdot\bn \,\det(\partial_y\bm{\varphi}_\zeta), 
\label{shellEQL}
\\
\partial_t \bU + (\bv\cdot \nabx) \bU
=
\mathbb{W}(\nabx \bv)\bU + \bU\mathbb{W}((\nabx \bv)^\top)   \label{soluteL}
\end{align}
which holds on $I\times\Ozeta\subset \mathbb R^{1+2}$(with \eqref{shellEQL} posed on $I\times\omega\subset \mathbb R^{1+1}$)  where
\begin{align*}
\mathbb{S}_*= (\nabx \bv +(\nabx \bv)^\top) -\pi\mathbb{I}+  \bU,
\qquad
\mathbb{W}(\nabx \bv)=\tfrac{1}{2}(\nabx \bv-(\nabx \bv)^\top).
\end{align*}
We complement \eqref{divfreeL}--\eqref{soluteL} with the following initial conditions
\begin{align}
&\zeta(0,\cdot)=\zeta_0(\cdot), \qquad\partial_t\zeta(0,\cdot)=\zeta_\star(\cdot) & \text{in }\omega,
\\
&\bv(0,\cdot)=\bv_0(\cdot) & \text{in }\Omega_{\zeta_0},
\\
&\bU(0,\cdot)=\bU_0(\cdot) &\text{in }\Omega_{\zeta_0},
\label{initialCondSolvL}  
\end{align}
and again, we impose periodicity on the boundary of $\omega$ and the following interface condition
\begin{align} 
\label{interfaceL}
&\bv\circ\bm{\varphi}_\zeta=(\partial_t\zeta)\bn & \text{on }I\times \omega
\end{align}
at the flexible part of the boundary with normal $\bn$. 
A similar system was analysed in \cite[(1.13)-(1.17)]{mensah2025vanishing} with an extra equation for the evolution of the polymer number density $\sigma$ and with an additional damping term $2(\bU-\sigma\mathbb{I})$ in the equation for $\bU$. For \textit{dilute} polymers, monomer-to-monomer interactions are negligible and thus it is reasonable to assume $\sigma$ is a constant. In this case, there is no evolutionary equation for $\sigma$ and $\bU=\widehat{\bU}-\sigma\mathbb{I}$ satisfy \eqref{soluteL}, where
\begin{align*} 
\widehat{\bU}(t, \bx)= \int_{B} f(t,\bx,\bq)\bq\otimes\bq \dq, \qquad \varepsilon=\frac{1}{2\lambda}=0. 
\end{align*}
Moreover, from the purely analytical point of view, the equation for $\sigma$ which is a homogeneous transport equation does not impose any analytical difficulty that is not already captured by the equation for $\bU$. As a result, many analysts, see for instant \cite{chupin2015stationary, hinch2021oldroyd, wang2024}, elect to study the system without $\sigma$. Putting this minor difference aside, the main difference between \eqref{divfreeL}-\eqref{soluteL} and the limit system analysed in \cite[(1.13)-(1.17)]{mensah2025vanishing} is that the former does not have the damping term $2(\bU-\sigma\mathbb{I})$ (or just $2\bU$ when the polymer number density is a constant) present in the latter. This term is a hyperbolic dissipation term that regularises the equation for the solute (without needing a more regular initial condition) from the mathematical point of view. From the physical point of view, this captures the system with no localised damping in the solute.

Furthermore, the purely hyperbolic equation \eqref{soluteL} has `truly' infinite memory since its  stress never relaxes. A consequence of this is that there is no exponential decay of stress and thus, Theorem \ref{thm:one} or any variant of it is completely out of reach. Nevertheless, as pointed out earlier, our second main result will study the singular limit of the system \eqref{divfree}-\eqref{solute} with finite memory to the system \eqref{divfreeL}-\eqref{soluteL} with this truly infinite memory. It is worth noting that although the limit system in \cite[(1.13)-(1.17)]{mensah2025vanishing} also has infinite memory, this is weaker in the sense that it's memory is exponentially fading but nonzero for all past times. 
%
Finally, we also note that,  whereas \cite[(1.13)-(1.17)]{mensah2025vanishing}  is analytically delicate yet manageable to handle, the former is extremely hard and becomes comparable in difficulty to the 3-D Euler equations when the full velocity gradient is considered in the solute equation (rather than the anti-symmetric gradient).

We are interested in weak solutions of \eqref{divfreeL}--\eqref{interfaceL} where weak solutions are defined as follows:
\begin{definition}[Weak solution]
\label{def:weaksolmartFP}
Let $( \bU_0, \bv_0, \zeta_0, \zeta_\star)$
be a dataset that satisfies
\begin{equation}
\begin{aligned}
\label{mainDataForAll}
&
\zeta_0 \in W^{2,2}(\omega) \text{ with } \Vert \zeta_0 \Vert_{L^\infty( \omega)} < L, \quad \zeta_\star \in L^{2}(\omega),
\\&\bv_0 \in L^{2}_{\divx}(\Omega_{\zeta_0} )\text{ is such that }\bv_0 \circ \bm{\varphi}_{\zeta_0} =\zeta_\star \bn \text{ on } \omega,
\\& 
\bU_0\in L^{2}(\Omega_{\zeta_0}).
\end{aligned}
\end{equation}
We call
$(\zeta, \bv,  \bU)$  
a \textit{weak solution} of   \eqref{divfreeL}--\eqref{interfaceL} with dataset $( \bU_0, \bv_0, \zeta_0, \zeta_\star)$  if: 
\begin{itemize}
\item the following properties 
\begin{align*}
&\zeta\in  W^{1,\infty}\big(I;L^{2}(\omega)  \big) \cap W^{1,2}(I;W^{1,2}(\omega))  \cap L^{\infty}\big(I;W^{2,2}(\omega)  \big),
\\&
\Vert\zeta\Vert_{L^\infty(I\times\omega)}<L,
\\
&\bv\in
L^{\infty} \big(I; L^{2}(\Ozeta) \big)\cap L^2\big(I;W^{1,2}_{\divx}(\Ozeta)  \big), 
\\
&
\bU \in   L^{\infty}\big(I;L^{2}(\Ozeta)  \big)  
\end{align*}
holds; 
\item for all  $  \mathbb{Y}  \in C^\infty (\overline{I}\times \R^2  )$, we have
\begin{equation}
\begin{aligned} 
\label{weakFokkerPlanckEq}
\int_I  \frac{\mathrm{d}}{\dt}
\int_{\Ozeta } \bU:\mathbb{Y} \dx \dt 
&=
\int_I
\int_{\Omega_{\zeta }}[\bU :\partial_t\mathbb{Y} + \bU:(\bv\cdot\nabx) \mathbb{Y}] \dx\dt
\\&+
\int_I\int_{\Ozeta}
[(\nabx \bv )\bU  + \bU (\nabx \bv )^\top ]:\mathbb{Y} \dx\dt;
\end{aligned}
\end{equation}
\item for all  $(\bm{\phi},\phi)  \in C^\infty_{\divx} (\overline{I}\times \R^2  )\otimes  C^\infty (\overline{I}\times \omega )$ with $\bm{\phi}(T,\cdot)=0$, $\phi(T,\cdot)=0$ and $\bm{\phi}\circ \bm{\varphi}_\zeta= \phi \bn$, we have
\begin{equation}
\begin{aligned}
\label{weakFluidStrut}
\int_I  \frac{\mathrm{d}}{\dt}\bigg(
\int_{\Ozeta } \bv\cdot\bm{\phi} \dx 
+
\int_\omega\partial_t\zeta\phi\dy\bigg) \dt 
&=
\int_I
\int_{\Ozeta }[\bv \cdot\partial_t\bm{\phi} + \bv\cdot(\bv\cdot\nabx) \bm{\phi}] \dx\dt
\\&-
\int_I\int_{\Ozeta }\big[\nu\nabx \bv  :\nabx \bm{\phi} +\bU:\nabx\bm{\phi}\big] \dx\dt
\\&
+
\int_I\int_\omega\big[\partial_t\zeta\partial_t\phi 
-
\gamma
\partial_t\naby\zeta\naby\phi 
-
\naby^2\zeta\naby^2\phi \big]\dy\dt;
\end{aligned}
\end{equation} 
\end{itemize}
\end{definition}

 
A superficial disadvantage of the system  \eqref{divfreeL}-\eqref{soluteL} without any form of diffusion or damping is that the regularisation effect due to those two terms is lost. Consequently, the solute component $ \bU$ of its weak solution is less regular than the 
solute component $\bT$ of the weak solution of
\eqref{divfree}-\eqref{solute}, c.f. \cite{mensah2023weak}. Nevertheless, because  $\bU$ possesses transport and corotational properties, any weak solution automatically regularises in the sense that it becomes essentially bounded in both space and time provided the initial condition is essentially  bounded in space. Indeed, let $(\zeta, \bv,  \bU)$  
be a weak solution of   \eqref{divfreeL}--\eqref{interfaceL} with dataset $( \bU_0, \bv_0, \zeta_0, \zeta_\star)$ satisfying \eqref{mainDataForAll} and $ \bU_0\in L^\infty(\Omega_{\zeta_0})$.
Although Equation \eqref{soluteL} looks complicated, we can show that it conserves all $L^p$-norms uniformly in $p\in [1,\infty)$ and thus, it is  essentially bounded in spacetime. To see the formal details, we test   \eqref{soluteL} with $\bU^{q-1}$, $q\in[1,\infty)$ and use Proposition \ref{prop:zeroCorotational}. We obtain
\begin{align}\label{conserveSolute}
\Vert \bU(t)\Vert_{L^q(\Omega_{\zeta})}^q
=
\Vert \bU_0\Vert_{L^q(\Omega_{\zeta_0})}^q 
\end{align}
uniformly in $q\in[1,\infty)$. So, in particular, we obtain $\bU \in L^\infty(I;L^\infty(\Omega_{\zeta}))$ since $ \bU_0\in L^\infty(\Omega_{\zeta_0})$ by assumption. 

The discussion above motivates a finer notion (for the solute component) of a weak solution of   \eqref{divfreeL}--\eqref{interfaceL}. We shall refer to this as an \textit{essentially bounded weak solution} whose precise definition is given as follows:
\begin{definition}[Essentially bounded weak solution]
\label{def:weaksolmartFPBounded}
Let $(  \bU_0, \bv_0, \zeta_0, \zeta_\star)$
be a dataset that satisfies \eqref{mainDataForAll} and
\begin{equation}
\begin{aligned}
\label{mainDataForAllBounded}
 \bU_0\in L^{\infty}(\Omega_{\zeta_0}).
\end{aligned}
\end{equation}
We call
$(\zeta, \bv,  \bU)$  
an \textit{essentially bounded weak solution} of   \eqref{divfreeL}--\eqref{interfaceL} with dataset $(   \bU_0, \bv_0, \zeta_0, \zeta_\star)$  if:
\begin{itemize}
\item $(\zeta, \bv, \bU)$  
is a  weak solution of   \eqref{divfreeL}--\eqref{interfaceL} with dataset $( \bU_0, \bv_0, \zeta_0, \zeta_\star)$;
\item the pair $ \bU $ satisfies $
 \bU \in L^\infty(I;L^\infty(\Omega_{\zeta})).
$
\end{itemize}
\end{definition}  
With Definition \ref{def:weaksolmartFPBounded} and the discussion preceding it in hand, the following theorem on the existence of an essentially bounded weak solution of   \eqref{divfreeL}--\eqref{interfaceL} is immediate and follows a similar argument as in \cite[Theorem 1.4]{mensah2025vanishing}.
\begin{theorem}
\label{thm:main} 
For a dataset $( \bU_0, \bv_0, \zeta_0, \zeta_\star)$ satisfying \eqref{mainDataForAll}
with $\bU_0$ satisfying \eqref{mainDataForAllBounded}, there exists an essentially bounded weak solution $(\zeta,\bv,  \bU)$  of \eqref{divfreeL}-\eqref{interfaceL}.
\end{theorem}
Our second main result will now seek to establish a relationship between strong solutions of \eqref{divfree}-\eqref{interface}  whose solution has finite memory 
 and essentially bounded weak solutions of   \eqref{divfreeL}--\eqref{interfaceL}  with truly infinite memory.  The precise statement is given as follows.
\begin{theorem}
\label{thm:main2}
Let $( \bU_0, \bv_0, \zeta_0, \zeta_\star)$ be a dataset satisfying \eqref{mainDataForAll} and \eqref{mainDataForAllBounded} and let $( {\bT}^\varepsilon_0, {\bu}^\varepsilon_0, \eta^\varepsilon_0, \eta^\varepsilon_\star)_{\varepsilon>0}$ be datasets satisfying \eqref{mainDataForAllStrong} and
\begin{equation}
\begin{aligned}
\label{dataConv}
&\eta^\varepsilon_0 \rightarrow \zeta_0 \quad\text{in}\quad W^{2,2}( \omega), 
\\&\eta^\varepsilon_\star \rightarrow \zeta_\star \quad\text{in}\quad L^\infty(I;L^{2}( \omega)), 
\\&\bm{1}_{\Omega_{\eta^\varepsilon_0}} \bu^\varepsilon_0  \rightarrow \bm{1}_{\Omega_{\zeta_0}}\bv_0 \quad\text{in}\quad L^2(\Omega \cup S_\ell),   
\\&\bm{1}_{\Omega_{\eta^\varepsilon_0}}
\bT^\varepsilon_0  \rightarrow \bm{1}_{\Omega_{\zeta_0}}\bU_0 \quad\text{in}\quad L^2(\Omega \cup S_\ell)
\end{aligned}
\end{equation}
as $\varepsilon\rightarrow0$.  
Now let $(\eta^\varepsilon, {\bu}^\varepsilon, {p}^\varepsilon,  {\bT}^\varepsilon)_{\varepsilon>0}$ be a collection of strong solutions of \eqref{divfree}-\eqref{interface} with dataset $(  {\bT}^\varepsilon_0, {\bu}^\varepsilon_0, \eta^\varepsilon_0, \eta^\varepsilon_\star)_{\varepsilon>0}$ and let $(\zeta,\bv,  \bU)$ be  an essentially bounded weak solution of \eqref{divfreeL}-\eqref{interfaceL} with dataset $(  \bU_0, \bv_0, \zeta_0, \zeta_\star)$. Then the following convergences
\begin{align*}
&\eta^\varepsilon \rightarrow \zeta \quad\text{in}\quad L^\infty(I;W^{2,2}( \omega)), 
\\&\partial_t\eta^\varepsilon \rightarrow \partial_t\zeta \quad\text{in}\quad L^\infty(I;L^{2}( \omega)),
\\&\partial_t\eta^\varepsilon \rightarrow \partial_t\zeta \quad\text{in}\quad L^2(I;W^{1,2}( \omega)),
\\&\bm{1}_{\Omega_{\eta^\varepsilon}} \bu^\varepsilon  \rightarrow \bm{1}_{\Omega_{\zeta}}\bv \quad\text{in}\quad L^\infty(I;L^2(\Omega \cup S_\ell)), 
\\&\bm{1}_{\Omega_{\eta^\varepsilon}}\nabx\bu^\varepsilon  \rightarrow \bm{1}_{\Omega_{\zeta}}\nabx\bv \quad\text{in}\quad L^2(I;L^2(\Omega \cup S_\ell)),  
\\&\bm{1}_{\Omega_{\eta^\varepsilon}}
\bT^\varepsilon  \rightarrow \bm{1}_{\Omega_{\zeta}}\bU \quad\text{in}\quad L^\infty(I;L^2(\Omega \cup S_\ell))
\end{align*}
hold. 
\end{theorem}
\begin{remark}
Theorem \ref{thm:main2} trivially holds for fixed spatial domains where formally speaking, $\eta^\varepsilon=\zeta=0$ and $\Omega_{\eta^\varepsilon}=\Ozeta=\Omega$.
\end{remark} 
An immediate consequence of Theorem \ref{thm:main2} is a \textit{weak-strong uniqueness}  result. Generally speaking, such a result states that  a weak solution is unique in the class of a strong solution.
This will usually require that both solutions solves the same equation and may also require additionally regularity of either solution or their data. See for instant \cite{BMSS, chemetov2019weak, germain2011weak, schwarzacher2022weak}. 
In our setting, however, \eqref{divfree}--\eqref{interface} and \eqref{divfreeL}--\eqref{interfaceL}  are two different systems, albeit related, with their own notion of a solution. Nevertheless, we present a \textit{`limiting' weak-strong uniqueness} result which states that as $\varepsilon\rightarrow0$, the essentially bounded weak solution to \eqref{divfreeL}--\eqref{interfaceL} will be unique in the class of the strong solution to \eqref{divfree}--\eqref{interface}. Here, we do not require additional regularity assumption on their respective dataset. The precise statement is:
 \begin{proposition}
Let $(\eta,\bu, p, \bT)$  be a strong solution of \eqref{divfree}-\eqref{interface} with dataset $(  \bT_0, \bu_0, \eta_0, \eta_\star)$ satisfying \eqref{mainDataForAllStrong} and let $(\zeta,\bv,\bU)$ be an essentially bounded weak solution of \eqref{divfreeL}-\eqref{interfaceL} with dataset $(  \bU_0, \bv_0, \zeta_0, \zeta_\star)$ satisfying \eqref{mainDataForAll} and \eqref{mainDataForAllBounded}.
Define
\begin{align*}
\overline{\bu} =\bu \circ \bm{\Psi}_{\eta-\zeta}, \quad \overline{p} =p \circ \bm{\Psi}_{\eta-\zeta},\quad  \overline{\bT} =\bT \circ \bm{\Psi}_{\eta-\zeta}
\end{align*}
with respect to the  Hanzawa transform $\bm{\Psi}_{\eta-\zeta}$. 
As $\varepsilon \rightarrow0$, we have
\begin{equation}
\begin{aligned} 
\sup_{t'\in (0,t]}&\big(
 \Vert  \partial_t(\eta-\zeta)(t') \Vert_{L^{2}(\omega )}^2
+
\Vert  \naby^2(\eta-\zeta)(t') \Vert_{L^{2}(\omega )}^2 
+
\Vert (\overline{\bu}-\bv)(t')\Vert_{L^2(\Omega_{\zeta})}^2
\big)
\\&
+\sup_{t'\in (0,t]}  
\Vert (\overline{\bT}-\bU)(t')\Vert_{L^2(\Omega_{\zeta})}^2 
+ 
\int_0^t
\big(
\Vert \nabx(\overline{\bu}-\bv)\Vert_{L^2(\Omega_{\zeta})}^2 
+ 
\gamma  \Vert  \partial_{t'}\naby(\eta-\zeta) \Vert_{L^{2}(\omega )}^2
\big)
\dt'  
\\
\lesssim& 
 \Vert   \eta_\star-\zeta_\star \Vert_{L^{2}(\omega )}^2
+
\Vert  \naby^2(\eta_0-\zeta_0) \Vert_{L^{2}(\omega )}^2  
+
\Vert \overline{\bu}_0-\bv_0\Vert_{L^2(\Omega_{\zeta_0})}^2
+
\Vert \overline{\bT}_0-\bU_{0}\Vert_{L^2(\Omega_{\zeta_0})}^2 
\end{aligned}
\end{equation}
for all $t\in I$ with a constant that depends on  $\gamma$ and $T$.
\end{proposition}
\begin{proof}
As mentioned earlier, the proof of this result is a direct consequence of Theorem \ref{thm:main2} and in particular, achieved by passing to the limit $\varepsilon\rightarrow0$ in \eqref{contrEst0} of Proposition \ref{prop:main} below.
\end{proof}
\begin{remark}
Unfortunately we are unable to recover actual unconditional uniqueness for essentially bounded weak solutions of \eqref{divfreeL}-\eqref{interfaceL} with dataset   satisfying \eqref{mainDataForAll} and \eqref{mainDataForAllBounded} from our analysis and it also not clear if that is to be expected.  
\end{remark}

\subsection{Plan for the rest of the paper}
In the next section, Section \ref{sec:prelim}, we collect some notations, set up our functional framework, and collect some key results that would be used repeatedly in the proof of our main results. We then devote Section \ref{sec:decay} and Section \ref{sec:singularLimit} to the proofs of Theorem \ref{thm:one} and Theorem \ref{thm:main2}, respectively.

\section{Preliminaries }
\label{sec:prelim}
\noindent  
For two non-negative quantities $F$ and $G$, we write $F \lesssim G$  if there is a constant $c>0$  such that $F \leq c\,G$. If $F \lesssim G$ and $G\lesssim F$ both hold, we use the notation $F\sim G$.  The anti-symmetric gradient of a vector $\bff\in \mathbb{R}^d$ is denoted by $\mathbb{W}(\nabx \bff)=\tfrac{1}{2}(\nabx \bff-(\nabx \bff)^\top)$ and the scalar matrix product of the matrices $\mathbb{A}=(a_{ij})_{i,j=1}^d$ and $\mathbb{B}=(b_{ij})_{i,j=1}^d$ is denoted by $\mathbb{A}:\mathbb{B}=\sum_{ij}a_{ij}b_{ij}$.
The symbol $\vert \cdot \vert$ may be used in four different contexts. For a scalar function $f\in \mathbb{R}$, $\vert f\vert$ denotes the absolute value of $f$. For a vector $\bff\in \mathbb{R}^d$, $\vert \bff \vert$ denotes the Euclidean norm of $\bff$. For a square matrix $\mathbb{F}\in \mathbb{R}^{d\times d}$, $\vert \mathbb{F} \vert$ shall denote the Frobenius norm $\sqrt{\mathrm{trace}(\mathbb{F}^T\mathbb{F})}$. Also, if $S\subseteq  \mathbb{R}^d$ is  a (sub)set, then $\vert S \vert$ is the $d$-dimensional Lebesgue measure of $S$.  
Since we only consider functions on $\omega \subset\mathbb{R}$ with periodic boundary
conditions and zero mean values, we have the following equivalences
\begin{align}
\label{equiNorm}
\Vert \cdot\Vert_{W^{1,2}(\omega)}\sim
\Vert \partial_y\cdot\Vert_{L^{2}(\omega)},
\qquad
\Vert \cdot\Vert_{W^{2,2}(\omega)}\sim
\Vert \partial_y^2\cdot\Vert_{L^{2}(\omega)},
\qquad
\Vert \cdot\Vert_{W^{4,2}(\omega)}\sim
\Vert \partial_y^4\cdot\Vert_{L^{2}(\omega)}.
\end{align} 
For $I:=(0,T)$, $T>0$, and $\eta\in C(\overline{I}\times\omega)$ satisfying $\|\eta\|_{L^\infty(I\times\omega)}\leq L$ where $L>0$ is a constant, we define for $1\leq p,r\leq\infty$,
\begin{align*} 
L^p(I;L^r(\Omega_\eta))&:=\Big\{v\in L^1(I\times\Omega_\eta):\substack{v(t,\cdot)\in L^r(\Omega_{\eta(t)})\,\,\text{for a.e. }t,\\\|v(t,\cdot)\|_{L^r(\Omega_{\eta(t)})}\in L^p(I)}\Big\},\\
L^p(I;W^{1,r}(\Omega_\eta))&:=\big\{v\in L^p(I;L^r(\Omega_\eta)):\,\,\nabx v\in L^p(I;L^r(\Omega_\eta))\big\}.
\end{align*} 
Higher-order Sobolev spaces can be defined accordingly. For $k>0$ with $k\notin\mathbb N$, we define the fractional Sobolev space $L^p(I;W^{k,r}(\Oeta))$ as the class of $L^p(I;L^r(\Omega_\eta))$-functions $v$ for which 
\begin{align*}
\|v\|_{L^p(I;W^{k,r}(\Oeta))}^p
&=\int_I\bigg(\int_{\Oeta} \vert v\vert^r\dx
+\int_{\Oeta}\int_{\Oeta}\frac{|v(\bx)-v(\bx')|^r}{|\bx-\bx'|^{d+k r}}\dx\dx'\bigg)^{\frac{p}{r}}\dt
\end{align*}
is finite. Accordingly, we can also introduce fractional differentiability in time for the spaces on moving domains.

Next, 
the Hanzawa transform $\bm{\Psi}_\eta$ together with its inverse 
 $\bm{\Psi}_\eta^{-1} : \Oeta \rightarrow\Omega$  possesses the following properties, see \cite{breit2022regularity, BMSS} for details. If for some $\ell,R>0$, we assume that
\begin{align*}
\Vert\eta\Vert_{L^\infty(\omega)}
+
\Vert\zeta\Vert_{L^\infty(\omega)}
< \ell <L \qquad\text{and}\qquad
\Vert\naby\eta\Vert_{L^\infty(\omega)}
+
\Vert\naby\zeta\Vert_{L^\infty(\omega)}
<R
\end{align*}
holds, then for any  $s>0$, $\varrho,p\in[1,\infty]$ and for any $\eta,\zeta \in B^{s}_{\varrho,p}(\omega)\cap W^{1,\infty}(\omega)$ (where $B^{s}_{\varrho,p}$ is a Besov space), we have that the estimates
\begin{align}
\label{210and212}
&\Vert \bm{\Psi}_\eta \Vert_{B^s_{\varrho,p}(\Omega\cup S_\ell)}
+
\Vert \bm{\Psi}_\eta^{-1} \Vert_{B^s_{\varrho,p}(\Omega\cup S_\ell)}
 \lesssim
1+ \Vert \eta \Vert_{B^s_{\varrho,p}(\omega)},
\\
\label{211and213}
&\Vert \bm{\Psi}_\eta - \bm{\Psi}_\zeta  \Vert_{B^s_{\varrho,p}(\Omega\cup S_\ell)} 
+
\Vert \bm{\Psi}_\eta^{-1} - \bm{\Psi}_\zeta^{-1}  \Vert_{B^s_{\varrho,p}(\Omega\cup S_\ell)} 
\lesssim
 \Vert \eta - \zeta \Vert_{B^s_{\varrho,p}(\omega)}
\end{align}
and
\begin{align}
\label{218}
&\Vert \partial_t\bm{\Psi}_\eta \Vert_{B^s_{\varrho,p}(\Omega\cup S_\ell)}
\lesssim
 \Vert \partial_t\eta \Vert_{B^{s}_{ \varrho,p}(\omega)},
\qquad
\eta
\in W^{1,1}(I;B^{s}_{\varrho,p}(\omega))
\end{align}
holds uniformly in time with the hidden constants depending only on the reference geometry, on $L-\ell$ and $R$. Finally, we present a useful result we shall use at several instances in our proof. 
\begin{proposition}
\label{prop:zeroCorotational}
Let $\bw=(w_1,w_2)$ be a $2$-d vector and $\mathbb{Z}=(z_{ij})_{i,j=1}^2$ a $2\times 2$ matrix. Then for any $n\in\mathbb{N}$ and for any $\mathbb{Y}\in\{\mathbb{Z},\mathbb{Z}^\top\}$, the equation
\begin{align*}
\mathbb{W}(\nabx \bw)\mathbb{Z}:\mathbb{Y}
^n + \mathbb{Z}\mathbb{W}((\nabx \bw)^\top) :\mathbb{Y}^n=0   
\end{align*}
holds where $\mathbb{Y}^n$ is the $n$-times matrix multiplication.
\end{proposition} 
The proof of Proposition \ref{prop:zeroCorotational} is given in \cite{mensah2025vanishing}.

\section{Decay rate}
\label{sec:decay}
This short section is devoted to the proof of our first main result, Theorem \ref{thm:one}. 
%
%
We begin by finding a rate for the extra stress tensor. For this, testing \eqref{solute} with $\bT$ and using Proposition \ref{prop:zeroCorotational} yield
\begin{align*}
\frac{1}{2}\frac{\dd}{\dt}\Vert \bT \Vert_{L^2(\Oeta)}^2
+
\varepsilon\Vert \bT \Vert_{L^2(\Oeta)}^2
+
\varepsilon \Vert\nabx \bT \Vert_{L^2(\Oeta)}^2=0 
\end{align*}
so we can deduce that
\begin{align*}
\frac{\dd}{\dt}\Vert \bT \Vert_{L^2(\Oeta)}^2
+
2\varepsilon\Vert \bT \Vert_{L^2(\Oeta)}^2
\leq 0 .
\end{align*}
Consequently, applying Gr\"onwall's lemma yield
\begin{align}
\label{t2}
\Vert \bT(t)\Vert_{L^2(\Oeta)}^2\leq e^{-2\varepsilon t}\Vert\bT_0\Vert_{L^2(\Omega_{\eta_0})}^2.
\end{align}
Taking the square-root yields the first estimate in Theorem \ref{thm:one}.

Next we test the solvent-structure subsystem \eqref{momEq}-\eqref{shellEQ} with $(\bu, \partial_t\eta)$. Due to the interface condition \eqref{interface}, we obtain
\begin{align*}
\frac{1}{2}\frac{\dd}{\dt}\big(
\Vert \partial_t\eta\Vert_{L^2(\omega)}^2
+
\Vert \naby^2\eta\Vert_{L^2(\omega)}^2
+
\Vert \bu \Vert_{L^2(\Oeta)}^2)
+
\big(
\gamma\Vert \partial_t\naby\eta \Vert_{L^2(\omega)}^2
+
\nu \Vert\nabx \bu \Vert_{L^2(\Oeta)}^2\big)
\\
\leq \frac{\nu}{2}\Vert\nabx \bu \Vert_{L^2(\Oeta)}^2+
\frac{1}{2\nu}
\Vert  \bT \Vert_{L^2(\Oeta)}^2.
\end{align*}
If we now absorb the viscous term on the right-hand side into the corresponding term on the left, then in particular,
\begin{align*}
&\frac{1}{2}\frac{\dd}{\dt} 
\Vert \partial_t\eta\Vert_{L^2(\omega)}^2 
+ 
\gamma\Vert \partial_t\naby\eta \Vert_{L^2(\omega)}^2
\leq 
\frac{1}{2\nu}
\Vert  \bT \Vert_{L^2(\Oeta)}^2,
\\
&\frac{1}{2}\frac{\dd}{\dt} 
\Vert \bu \Vert_{L^2(\Oeta)}^2 
+ 
\frac{\nu}{2} \Vert\nabx \bu \Vert_{L^2(\Oeta)}^2  
\leq \frac{1}{2\nu}
\Vert  \bT \Vert_{L^2(\Oeta)}^2.
\end{align*}
From the first estimate above, we use the classical Poincar\'e's inequality and \eqref{t2} to obtain
\begin{align*}
& \frac{\dd}{\dt} 
\Vert \partial_t\eta\Vert_{L^2(\omega)}^2 
+ 
2c\gamma\Vert \partial_t \eta \Vert_{L^2(\omega)}^2
\leq 
\frac{1}{\nu}
\Vert  \bT \Vert_{L^2(\Oeta)}^2 
\leq 
\frac{e^{-2\varepsilon t}}{\nu}
\Vert  \bT_0 \Vert_{L^2(\Omega_{\eta_0})}^2 
\end{align*}
which when integrated in time gives
\begin{align*} 
\Vert \partial_t\eta(t)\Vert_{L^2(\omega)}^2 
\leq 
\Vert  \eta_\star\Vert_{L^2(\omega)}^2 
+
\frac{1}{2\nu\varepsilon}(1-e^{-2\varepsilon t})
\Vert  \bT_0 \Vert_{L^2(\Omega_{\eta_0})}^2 
-
\int_0^t
2c\gamma\Vert \partial_s \eta \Vert_{L^2(\omega)}^2 \ds
\end{align*}
and thus by Gr\"onwall's lemma,
\begin{align*} 
\Vert \partial_t\eta(t)\Vert_{L^2(\omega)}^2 
\leq 
e^{-2c\gamma t}
\left[
\Vert  \eta_\star\Vert_{L^2(\omega)}^2 
+
\frac{1}{2\nu\varepsilon}(1-e^{-2\varepsilon t})
\Vert  \bT_0 \Vert_{L^2(\Omega_{\eta_0})}^2 
\right].
\end{align*}
Taking the square-root yields the second estimate in Theorem \ref{thm:one}. We note in passing that  $(1-e^{-2\varepsilon t})\leq 1$.
Let now return to
\begin{align*}
\frac{\dd}{\dt} 
\Vert \bu \Vert_{L^2(\Oeta)}^2 
+ 
\nu\Vert\nabx \bu \Vert_{L^2(\Oeta)}^2  
\leq \frac{1}{\nu}
\Vert  \bT \Vert_{L^2(\Oeta)}^2.
\end{align*}
By using the reverse triangle inequality and the Poincar\'e--Wirtinger inequality,
\begin{align*}
\frac{\dd}{\dt} 
\Vert \bu \Vert_{L^2(\Oeta)}^2 
+ 
c\nu \big(\Vert  \bu \Vert_{L^2(\Oeta)}
-
\Vert  \bu_{\Oeta} \Vert_{L^2(\Oeta)}
\big)^2  
\leq 
\frac{e^{-2\varepsilon t}}{\nu}
\Vert\bT_0\Vert_{L^2(\Omega_{\eta_0})}^2.
\end{align*}
However, by Jensen's inequality,  
\begin{align*}
\Vert  \bu_{\Oeta} \Vert_{L^2(\Oeta)}
\leq 
\Vert  \bu \Vert_{L^2(\Oeta)} \vert\Omega_\eta\vert^{-1/2}
\leq 
\Vert  \bu \Vert_{L^2(\Oeta)}\big( \inf_{t\in I}\vert\Omega_\eta\vert\big)^{-1/2}
\end{align*}
where 
\begin{align*}
\vert\Omega_\eta\vert=\int_\Omega 
\vert\det(\nabx \bm{\Psi}_\eta)\vert \dx 
\end{align*}
is the volume  of $\Omega_\eta$. Thus,
\begin{align*}
\Big(1-\big( \inf_{t\in I}\vert\Omega_\eta\vert\big)^{-1/2}\Big)\Vert  \bu \Vert_{L^2(\Oeta)}
\leq 
\Vert  \bu \Vert_{L^2(\Oeta)}
-
\Vert  \bu_{\Oeta} \Vert_{L^2(\Oeta)}
\end{align*}
and we obtain
\begin{align*}
\frac{\dd}{\dt} 
\Vert \bu \Vert_{L^2(\Oeta)}^2 
+ 
c\nu \Big(1-\big( \inf_{t\in I}\vert\Omega_\eta\vert\big)^{-1/2}\Big)^2\Vert  \bu \Vert_{L^2(\Oeta)}^2  
\leq 
\frac{e^{-2\varepsilon t}}{\nu}
\Vert\bT_0\Vert_{L^2(\Omega_{\eta_0})}^2.
\end{align*}
Consequently,  by Gr\"onwall's lemma,
\begin{align*} 
\Vert \bu(t) \Vert_{L^2(\Oeta)}^2 
\leq 
e^{- c\nu\left(1-\left( \inf_{t\in I}\vert\Omega_\eta\vert\right)^{-1/2}\right)^2 t}
\left[
\Vert \bu_0 \Vert_{L^2(\Omega_{\eta_0})}^2 
+
\frac{1}{2\nu\varepsilon}(1-e^{-2\varepsilon t})
\Vert  \bT_0 \Vert_{L^2(\Omega_{\eta_0})}^2 
\right]
\end{align*}
from which the desired result immediately follow.

\section{Vanishing limit}
\label{sec:singularLimit} 
We devote this section to the proof of our second main result, Theorem \ref{thm:main2}.
Our proof relies on a relative energy method where one measures the distance between two solutions in the energy norm. Unfortunately, because our two solutions under consideration are defined on two separate variable geometries, we are unable to directly measure their distance. To get around this issue, we transform one solution via the  Hanzawa transform onto the domain of the other and measure the distance on this latter domain. In this regard, we consider the following transformation 
\begin{align}
\label{transHan}
\overline{\bu} =\bu \circ \bm{\Psi}_{\eta-\zeta}, \quad \overline{p} =p \circ \bm{\Psi}_{\eta-\zeta},\quad  \overline{\bT} =\bT \circ \bm{\Psi}_{\eta-\zeta}
\end{align}
of  $(\eta ,\bu , p , \bT )$ onto the domain of  $(\zeta, \bv, \pi,  \bU)$ with respect to the Hanzawa transform $\bm{\Psi}_{\eta-\zeta}$. The proof of Theorem \ref{thm:main2} will now follow from the following proposition.
\begin{proposition}\label{prop:main}
Let $(\eta,\bu, p, \bT)$  be a strong solution of \eqref{divfree}-\eqref{interface} with dataset $(  \bT_0, \bu_0, \eta_0, \eta_\star)$. If $(\zeta,\bv, \bU)$ is an essentially bounded weak solution of \eqref{divfreeL}-\eqref{interfaceL} with dataset $(  \bU_0, \bv_0, \zeta_0, \zeta_\star)$, then we have
\begin{equation}
\begin{aligned}
\label{contrEst0} 
\sup_{t'\in (0,t]}&\big(
\Vert (\overline{\bu}-\bv)(t')\Vert_{L^2(\Omega_{\zeta})}^2
+
 \Vert  \partial_t(\eta-\zeta)(t') \Vert_{L^{2}(\omega )}^2
+
\Vert  \naby^2(\eta-\zeta)(t') \Vert_{L^{2}(\omega )}^2 
\big)
\\&
+\sup_{t'\in (0,t]} 
\Vert (\overline{\bT}-\bU)(t')\Vert_{L^2(\Omega_{\zeta})}^2  
+ 
\int_0^t
\big(
\Vert \nabx(\overline{\bu}-\bv)\Vert_{L^2(\Omega_{\zeta})}^2\dt'
+ 
\gamma 
\Vert  \partial_{t'}\naby(\eta-\zeta) \Vert_{L^{2}(\omega )}^2
\big)
\dt'  
\\
\leq&K
\bigg[ 
 \Vert   \eta_\star-\zeta_\star \Vert_{L^{2}(\omega )}^2
+
\Vert  \naby^2(\eta_0-\zeta_0) \Vert_{L^{2}(\omega )}^2  
+
\Vert \overline{\bu}_0-\bv_0\Vert_{L^2(\Omega_{\zeta_0})}^2
\\&
\qquad+
\Vert \overline{\bT}_0-\bU_{0}\Vert_{L^2(\Omega_{\zeta_0})}^2
+
\varepsilon^2\int_0^t
\big(\Vert \overline{\bT} \Vert_{W^{2,2}( \Omega_{\zeta})}^2 
+
\Vert \bU \Vert_{L^{\infty}( \Omega_{\zeta})}^2
\big)
\dt'
\bigg].
\end{aligned}
\end{equation}
for all $t\in I$ where $(\overline{\bu}, \overline{p},  \overline{\bT})$ is given by \eqref{transHan} and where
\begin{align*}
K:=&C\,
e^{
c\int_0^t
 \left( 1+c(\gamma)
+
\Vert \partial_{t'}\overline{\bu}  \Vert_{L^2(\Omega_{\zeta})}^2
+
\Vert \overline{\bu}  \Vert_{W^{2,2}(\Ozeta)}^2
+\Vert \overline{p}  \Vert_{W^{1,2}(\Ozeta)}^2 
+
\Vert  \overline{\bT} \Vert_{W^{2,2}(\Omega_{\zeta})}^2
+
\Vert \partial_{t'}\overline{\bT} \Vert_{L^{2}( \Omega_{\zeta})}^2   
+
\Vert  \bU \Vert_{L^{\infty}(\Omega_{\zeta})}^2 
\right)\dt } 
\end{align*}
with $c$ and $C$ being constants independent of   $\varepsilon$.
\end{proposition}
\begin{proof}
As shown \cite[Section 4.1]{breit2022regularity} (see also \cite[Section 5]{BMSS}) for the solvent-structure subproblem and in \cite[Section 5.3]{mensah2023weak} for the solute subproblem, the transformed solution  $(\eta , \overline{\bu} , \overline{p} ,\overline{\rho} ,\overline{\bT} )$  solves
\begin{align}
\mathbb{B}_{\eta-\zeta}^\top : \nabx\overline{\bu}  = 0, \label{divfree2} 
\\
\partial_t \overline{\bu}   = \Delx\overline{\bu}   
-
\nabx\overline{p} + \divx\overline{\bT} 
+
 \mathbf{h}_{\eta-\zeta}(\overline{\bu} )
-
\divx   \mathbb{G}_{\eta-\zeta}(\overline{\bu} ,\overline{p} ,\overline{\bT} ), \label{momEq2}
\\
\partial_t^2  \eta - \gamma\partial_t\partial_y^2  \eta +  \partial_y^4 \eta 
=  -([
\mathbb{A}_{\eta-\zeta}\nabx\overline{\bu}   
-
\mathbb{B}_{\eta-\zeta}(\overline{p} -\overline{\bT} )]\bn_{\zeta})\circ  \bm{\varphi}_{\zeta}\cdot\bn \,\det(\partial_y\bm{\varphi}_\zeta), 
\label{shellEQ2}
\\
\partial_t\overline{\bT} 
+
\overline{\bu}  \cdot\nabx \overline{\bT}  
=
\mathbb{W}(\nabx\overline{\bu}) \overline{\bT} 
+
\overline{\bT} \mathbb{W}( (\nabx\overline{\bu} )^\top)
-
\varepsilon (1
-\Delx)\overline{\bT} 
-\varepsilon
\divx(\nabx \overline{\bT} (\mathbb{I}-\mathbb{A}_{\eta-\zeta}))
\nonumber
\\
+\varepsilon
(1-J_{\eta-\zeta})  \overline{\bT}
+
\mathbb{H}_{\eta-\zeta}(\overline{\bu} ,\overline{\bT} ) \label{solute2}
\end{align}
on $I \times \Omega_{\zeta}$ (with \eqref{shellEQ2} posed on $I\times\omega$) subject to the interface condition $\overline{\bu}\circ\bm{\varphi}_\zeta=(\partial_t\eta)\bn $ on $I\times \omega$. Here,
\begin{align*}
&J_{\eta-\zeta} :=\det(\nabx \bm{\Psi}_{\eta-\zeta}),
\\
&
\mathbb{B}_{\eta-\zeta}:=
J_{\eta-\zeta} \nabx \bm{\Psi}_{\eta-\zeta}^{-1}\circ \bm{\Psi}_{\eta-\zeta}, 
\\
&\mathbb{A}_{\eta-\zeta}:=
\mathbb{B}_{\eta-\zeta}(\nabx \bm{\Psi}_{\eta-\zeta}^{-1}\circ \bm{\Psi}_{\eta-\zeta})^\top,
\\&
\mathbb{H}_{\eta-\zeta}(\overline{\bu} ,\overline{\bT} )
=
(1-J_{\eta-\zeta})\partial_t\overline{\bT} 
-
J_{\eta-\zeta} \nabx \overline{\bT} \cdot\partial_t\bm{\Psi}_{\eta-\zeta}^{-1}\circ \bm{\Psi}_{\eta-\zeta}
+
\mathbb{W}(\nabx\overline{\bu}) (\mathbb{B}_{\eta-\zeta}-\mathbb{I})\overline{\bT} 
\\&
\qquad\qquad
+
\overline{\bT} (\mathbb{B}_{\eta-\zeta}-\mathbb{I})^\top\mathbb{W}( (\nabx\overline{\bu} )^\top)
+
\overline{\bu}  \cdot\nabx \overline{\bT}  (\mathbb{I}-\mathbb{B}_{\eta-\zeta})
 ,
\\&
\mathbb{G}_{\eta-\zeta}(\overline{\bu} ,\overline{p} ,\overline{\bT} )
=
(\mathbb{I} -\mathbb{A}_{\eta-\zeta}) \nabx\overline{\bu} 
-
(\mathbb{I} -\mathbb{B}_{\eta-\zeta})(\overline{p} -\overline{\bT} ),
\\&
 \mathbf{h}_{\eta-\zeta}(\overline{\bu} )
 =
 (1-J_{\eta-\zeta})\partial_t\overline{\bu} 
-
J_{\eta-\zeta} \nabx \overline{\bu} \partial_t\bm{\Psi}_{\eta-\zeta}^{-1}\circ \bm{\Psi}_{\eta-\zeta}
-
\nabx\overline{\bu} \mathbb{B}_{\eta-\zeta}\overline{\bu}.
\end{align*} 
Furthermore, by following a similar argument as \cite[Lemma 4.2]{breit2022regularity}, we have that $(\eta, \overline{\bu}, \overline{p}, \overline{\bT} )$ is a strong solution in the sense that it inherits on $\Ozeta$, the same regularity properties of $(\eta, \bu, p, \bT)$ in Definition \ref{def:strongSolution}. With this preparation in hand,
we are now going to obtain estimates for the differences
$
(\eta-\zeta ,  \overline{\bu}-\bv, \overline{p}-\pi, \overline{\bT} -\bU)
$ where we start with the components that describe the solvent-structure subproblem.
\subsection{Estimate for the solvent-structure interaction}
Since the estimate for the solvent-structure subsystem remains the same as analysed in \cite[Section 4.1]{mensah2025vanishing}, we have that
\begin{equation}
\begin{aligned}
\label{velShellDiff2}
\big(&\Vert (\overline{\bu}-\bv)(t)\Vert_{L^2(\Omega_{\zeta})}^2
+
 \Vert  \partial_t(\eta-\zeta)(t) \Vert_{L^{2}(\omega )}^2
+
\Vert  \naby^2(\eta-\zeta)(t) \Vert_{L^{2}(\omega )}^2 \big)
\\&\qquad+ 
\int_0^t
\Vert \nabx(\overline{\bu}-\bv)\Vert_{L^2(\Omega_{\zeta})}^2\dt'
+
\gamma 
\int_0^t\Vert  \partial_{t'}\naby(\eta-\zeta) \Vert_{L^{2}(\omega )}^2
\dt'
\\
\lesssim &
\Vert \overline{\bu}_0-\bv_0\Vert_{L^2(\Omega_{\zeta})}^2
+
 \Vert   \eta_\star-\zeta_\star \Vert_{L^{2}(\omega )}^2
+
\Vert  \naby^2(\eta_0-\zeta_0) \Vert_{L^{2}(\omega )}^2 
   \\&
+ 
\int_0^t
\Vert  \overline{\bT} - \bU\Vert_{L^2(\Ozeta)}^2\dt'
+ 
\int_0^t
\Vert \overline{\bu} - \bv\Vert_{L^2(\Ozeta)}^2\dt'
\\&
+ 
\int_0^t 
  \Vert \partial_{t'}(\eta-\zeta)\Vert_{L^{2}(\omega)}^2
   \Vert \overline{\bu} \Vert_{W^{2,2}(\Ozeta)}^2  
 \dt'
\\& 
+ 
 \int_0^t
\Vert \eta-\zeta\Vert_{W^{2,2}(\omega)}^2
\big( c(\gamma)+\Vert \partial_{t'}\overline{\bu}  \Vert_{L^2(\Omega_{\zeta})}^2
+
\Vert \overline{\bu}  \Vert_{W^{2,2}(\Ozeta)}^2 
\big)\dt' 
\\& 
+  \int_0^t
\Vert \eta-\zeta\Vert_{W^{2,2}(\omega)}^2
\big( \Vert \overline{p}  \Vert_{W^{1,2}(\Ozeta)}^2
+\Vert \overline{\bT}  \Vert_{W^{1,2}(\Ozeta)}^2
\big)\dt' 
\end{aligned}
\end{equation}
hold for any $t\in I$. 

\subsection{Estimate for the solute}
\label{sec:Estimate for the solute}
With the desired estimate \eqref{velShellDiff2} for the solvent-structure subproblem in the previous subsection, we are now going to obtain an estimate for  $\Vert (\overline{\bT}-\bU)(t)\Vert_{L^2(\Omega_{\zeta})}^2$.  First of all, we note that due to Proposition \ref{prop:zeroCorotational}, if we test \eqref{soluteL} with $\bU$, we obtain
\begin{align*}
\frac{1}{2}\frac{\dd}{\dt} \Vert  \bU(t) \Vert_{L^{2}( \Omega_{\zeta})}^2 
=
0.
\end{align*} 
If we also test \eqref{solute2} with $\overline{\bT}$, we obtain
\begin{align*}
\frac{1}{2}\frac{\dd}{\dt} \Vert  \overline{\bT}(t) \Vert_{L^{2}( \Omega_{\zeta})}^2 
=&
\frac{1}{2}
\int_{\partial \Omega_{\zeta}}(\bn\partial_t\zeta)\circ\bm{\varphi}_\zeta^{-1}\cdot\bn_\zeta\vert\overline{\bT}\vert^2\dd\mathcal{H}^1
-
\int_{\Omega_{\zeta}}(\overline{\bu}\cdot\nabx)\overline{\bT} :\overline{\bT}\dx
\\
&+
\int_{\Omega_{\zeta}} [\varepsilon\Delx \overline{\bT} 
-
\varepsilon
\divx(\nabx \overline{\bT} (\mathbb{I}-\mathbb{A}_{\eta-\zeta}))
+
\varepsilon
(1-J_{\eta-\zeta})  \overline{\bT}
+
\mathbb{H}_{\eta-\zeta}(\overline{\bu},\overline{\bT} )]: \overline{\bT}\dx
\\
&
 -\varepsilon
 \Vert  \bT \Vert_{L^{2}( \Omega_{\zeta})}^2.
\end{align*}
Furthermore, by Reynold's transport theorem,
\begin{align*}
- \frac{\dd}{\dt}\int_{\Omega_{\zeta}}\overline{\bT} :\bU\dx
=&
-
\int_{\partial \Omega_{\zeta}}(\bn\partial_t\zeta)\circ\bm{\varphi}_\zeta^{-1}\cdot\bn_\zeta\overline{\bT} :\bU\dd\mathcal{H}^1
-
\int_{\Omega_{\zeta}} \partial_t\bU:\overline{\bT}\dx
-
\int_{\Omega_{\zeta}}\partial_t\overline{\bT} :\bU\dx
\\
=&
-
\int_{\partial \Omega_{\zeta}}(\bn\partial_t\zeta)\circ\bm{\varphi}_\zeta^{-1}\cdot\bn_\zeta\overline{\bT} :\bU \dd\mathcal{H}^1
+
\int_{\Omega_{\zeta}}( \bv\cdot \nabx) \bU:\overline{\bT} \dx
\\&
-
\int_{\Omega_{\zeta}}[\mathbb{W}(\nabx\bv) \bU
+
  \bU\mathbb{W}((\nabx\bv )^\top)
]:\overline{\bT} \dx
\\&
+
\int_{\Omega_{\zeta}} [(\overline{\bu} \cdot \nabx) \overline{\bT}
+
\varepsilon \overline{\bT}]:\bU \dx
\\&-
\int_{\Omega_{\zeta}}[\mathbb{W}(\nabx\overline{\bu} )\overline{\bT} 
+
\overline{\bT} \mathbb{W}( (\nabx\overline{\bu} )^\top)
]:\bU\dx
\\&
-
\int_{\Omega_{\zeta}}[\varepsilon\Delx \overline{\bT} 
-
\varepsilon
\divx(\nabx \overline{\bT} (\mathbb{I}-\mathbb{A}_{\eta-\zeta}))
+
\varepsilon
(1-J_{\eta-\zeta})  \overline{\bT}
+
\mathbb{H}_{\eta-\zeta}(\overline{\bu} ,\overline{\bT})]:\bU\dx.
\end{align*}
However, by using the divergence-free condition \eqref{divfreeL} and the interface condition \eqref{interfaceL}, we observe that
\begin{equation}
\begin{aligned}
\label{j1j2X}
\int_{\Omega_{\zeta}}&\big[(\bv\cdot \nabx) \bU:\overline{\bT} 
+
(\overline{\bu} \cdot \nabx) \overline{\bT}:\bU
-
(\overline{\bu}\cdot\nabx)\overline{\bT}:\overline{\bT}
\big]\dx
\\&+
\frac{1}{2}
\int_{\partial \Omega_{\zeta}}(\bn\partial_t\zeta)\circ\bm{\varphi}_\zeta^{-1}\cdot\bn_\zeta\big(\vert\overline{\bT}\vert^2
-2\overline{\bT}:\bU
\big)
\dd\mathcal{H}^1
\\&
=
\int_{\Ozeta}((\overline{\bu} - \bv )\cdot \nabx) \overline{\bT}:  (\bU - \overline{\bT} )\dx.
\end{aligned}
\end{equation}
Because of the basic identity $\tfrac{1}{2}|a-b|^2=\tfrac{1}{2}|a|^2+\tfrac{1}{2}|b|^2-ab$,
if we combine all the information above, we obtain
\begin{equation}
\begin{aligned}
\label{diffSolu1}
\frac{1}{2}\frac{\dd}{\dt} \Vert  (\overline{\bT}-\bU)(t) \Vert_{L^{2}( \Omega_{\zeta})}^2  
= & 
\int_{\Ozeta}((\overline{\bu} - \bv )\cdot \nabx) \overline{\bT}:  (\bU - \overline{\bT} )\dx
\\&
 -
\int_{\Omega_{\zeta}}[\mathbb{W}(\nabx\bv) \bU
+
  \bU\mathbb{W}((\nabx\bv )^\top)
]: \overline{\bT}  \dx 
\\
&
-
\int_{\Omega_{\zeta}}[\mathbb{W}(\nabx \overline{\bu})\overline{\bT} + \overline{\bT}\mathbb{W}((\nabx \overline{\bu})^\top)]:\bU\dx
\\
&-
\int_{\Omega_{\zeta}}\varepsilon (1- \Delx) \overline{\bT} 
: (\overline{\bT}-\bU)\dx
\\&-
\int_{\Omega_{\zeta}} 
\varepsilon
\divx(\nabx \overline{\bT} (\mathbb{I}-\mathbb{A}_{\eta-\zeta}))
 : (\overline{\bT}-\bU)\dx
 \\&+
 \int_{\Omega_{\zeta}} 
 \varepsilon
(1-J_{\eta-\zeta})  \overline{\bT} : (\overline{\bT}-\bU)\dx
\\&+
\int_{\Omega_{\zeta}} 
\mathbb{H}_{\eta-\zeta}( \overline{\bu} ,\overline{\bT}): (\overline{\bT}-\bU)\dx
\\&=:K_1+\ldots+K_7.
\end{aligned}
\end{equation} 
The main differences between \eqref{diffSolu1} and \cite[(4.20)]{mensah2025vanishing} is the absence of the damping difference present on the left and side of \cite[(4.20)]{mensah2025vanishing} as well as the $\varepsilon$-parameter term $K_6$ in \eqref{diffSolu1} which does not also appear in \cite[(4.20)]{mensah2025vanishing}. Since we have optimised the analysis in \cite[(4.20)]{mensah2025vanishing} to avoid utilising the damping difference, we can essentially repeat the estimates that are common in our both settings for completeness. Furthermore,  the $\varepsilon$-parameter term $K_6$  will be fairly easy to deal with as we shall soon see. Moving on, we 
%
now note that due to Ladyzhenskaya's inequality and the fact that the term $\Vert\nabx\overline{\bT}(t)\Vert_{L^2(\Ozeta)}$ is essentially bounded in time, we obtain
\begin{equation}
\begin{aligned}
\label{k2}
\int_0^t
K_1\dt'  
\lesssim&\int_0^t\big(
\Vert \overline{\bT}-\bU\Vert_{L^2(\Omega_{\zeta})}
\Vert   \overline{\bu} - \bv \Vert_{L^2(\Omega_{\zeta})}^{1/2}\Vert \nabx( \overline{\bu} - \bv) \Vert_{L^2(\Omega_{\zeta})}^{1/2}
\\&\times
\Vert  \nabx\overline{\bT} \Vert_{L^{2}(\Omega_{\zeta})}^{1/2}
\Vert  \overline{\bT} \Vert_{W^{2,2}(\Omega_{\zeta})}^{1/2}\big)\dt'
\\
\leq&
\delta_1
\sup_{t'\in (0,t]}\Vert (\overline{\bT}-\bU)(t')\Vert_{L^2(\Omega_{\zeta})}^2
+
\delta_1
\int_0^t
\Vert \nabx( \overline{\bu} - \bv) \Vert_{L^2(\Omega_{\zeta})}^2\dt'
\\&+
c(\delta_1)
\int_0^t
 \Vert \overline{\bu} - \bv \Vert_{L^2(\Omega_{\zeta})}^2
\Vert  \overline{\bT} \Vert_{W^{2,2}(\Omega_{\zeta})}^2\dt'
\end{aligned}
\end{equation}
for any $\delta_1>0$ and for any $t\in I$.
To estimate $K_2$ and $K_3$,   we first use the identity $\mathbb{W}((\nabx\bu )^\top)=-\mathbb{W}(\nabx\bu)$ and the relation $(\mathbb{A}\mathbb{B}):\mathbb{C}=(\mathbb{C}\mathbb{A}):\mathbb{B}$ which holds for matrices $\mathbb{A}\in \mathbb{R}^{2\times2}$ and $\mathbb{B},\mathbb{C}\in \mathbb{R}^{2\times2}_{\mathrm{symm}}$ to rewrite the terms in $K_2$ and $K_3$ as  
\begin{align*}
[\mathbb{W}(\nabx\bv) \bU
&+
  \bU\mathbb{W}((\nabx\bv )^\top)
]: \overline{\bT}  
+
[\mathbb{W}(\nabx \overline{\bu})\overline{\bT} + \overline{\bT}\mathbb{W}((\nabx \overline{\bu})^\top)]:\bU  
\\=& 
\tfrac{1}{2}
\Big\{
\bU \nabx (\overline{\bu}-\bv):(\overline{\bT}-\bU) 
-
\bU (\nabx (\overline{\bu}-\bv))^\top:(\overline{\bT}-\bU) 
 \Big\}
 \\
 \quad-&\tfrac{1}{2}\Big\{
 (\overline{\bT}-\bU) \nabx (\overline{\bu}-\bv):\bU
 -
 (\overline{\bT}-\bU) (\nabx (\overline{\bu}-\bv))^\top:\bU
 \Big\}
\end{align*} 
Thus, it follows that
\begin{align*}
\int_0^t (K_2+K_3)\dt'
\lesssim&
\int_0^t \Vert \nabx( \overline{\bu} - \bv) \Vert_{L^2(\Omega_{\zeta})} 
 \Vert \overline{\bT} -\bU \Vert_{L^2(\Omega_{\zeta})} 
\Vert  \bU \Vert_{L^{\infty}(\Omega_{\zeta})} \dt'
\\
\leq
&\delta_2
\int_0^t
\Vert \nabx( \overline{\bu} - \bv) \Vert_{L^2(\Omega_{\zeta})}^2\dt'
+ 
c(\delta_2)
\int_0^t
 \Vert \overline{\bT} -\bU \Vert_{L^2(\Omega_{\zeta})}^2
\Vert  \bU \Vert_{L^{\infty}(\Omega_{\zeta})}^2\dt'.
\end{align*} 
hold for any $\delta_2>0$ and for any $t\in I$. Next, we also have that
\begin{align*} 
\int_0^t
K_4\dt'  
&\leq
\delta_4
\sup_{t'\in (0,t]}
\Vert (\overline{\bT}-\bU)(t') \Vert_{L^{2}( \Omega_{\zeta})}^2
+
c(\delta_4) \varepsilon^2\int_0^t
\Vert  \overline{\bT} \Vert_{W^{2,2}(\Omega_{\zeta})}^2\dt'
\end{align*} 
for any $\delta_4>0$ and for any $t\in I$.  Also, since $\Vert \overline{\bT}(t) \Vert_{W^{1,2}( \Omega_{\zeta})}^2$ is essentially bounded in time,
\begin{align*}  
\int_0^t
K_5\dt'
\lesssim&
\int_0^t
\Vert  \overline{\bT}-\bU \Vert_{L^2(\Omega_{\zeta})}
\Vert \varepsilon\overline{\bT} \Vert_{W^{2,2}( \Omega_{\zeta})}
\Vert \eta-\zeta \Vert_{W^{1,\infty}( \omega)}\dt'
\\&
+
\int_0^t
\Vert  \overline{\bT}-\bU \Vert_{L^\infty(\Omega_{\zeta})}
\Vert \varepsilon\overline{\bT} \Vert_{W^{1,2}( \Omega_{\zeta})}
\Vert \eta-\zeta \Vert_{W^{2,2}( \omega)}
\dt'
\\
\leq&
\delta_5
\sup_{t'\in (0,t]}\Vert (\overline{\bT}-\bU)(t')\Vert_{L^2(\Omega_{\zeta})}^2
+
c(\delta_5)\varepsilon^2
\int_0^t
\Vert \eta-\zeta \Vert_{W^{2,2}( \omega)}^2 
\Vert \overline{\bT} \Vert_{W^{2,2}( \Omega_{\zeta})}^2
\dt'
\\&
+
\delta_5
\sup_{t'\in (0,t]}
\Vert (\eta-\zeta)(t') \Vert_{W^{2,2}( \omega)}^2 
+
c(\delta_5) \varepsilon^2
\int_0^t
\big(
\Vert  \overline{\bT}\Vert_{W^{2,2}(\Omega_{\zeta})}^2 +\Vert \bU \Vert_{L^\infty(\Omega_{\zeta})}^2\big)  \dt'
\end{align*}
hold for any $\delta_5>0$ and for any $t\in I$. We also observe that
\begin{align*}
\int_0^t
K_6\dt' 
&\lesssim
\int_0^t
\Vert \eta-\zeta \Vert_{W^{1,2}(\omega)} 
\Vert  \varepsilon(\overline{\bT}-\bU)\Vert_{L^\infty(\Ozeta)} 
\Vert   \overline{\bT}\Vert_{L^2(\Ozeta)} 
\dt'
\\&
\lesssim
\varepsilon^2
\int_0^t
\big(\Vert  \overline{\bT}\Vert_{W^{2,2}( \Omega_{\zeta})}^2 + \Vert \bU  \Vert_{L^{\infty}( \Omega_{\zeta})}^2\big)\dt'
+ 
\int_0^t
\Vert \eta-\zeta \Vert_{W^{1,2}(\omega)}^2 
\Vert \overline{\bT}\Vert_{L^2(\Ozeta)}^2 
\dt'.
\end{align*}
For $K_7$, we use the definition of $\mathbb{H}_{\eta-\zeta}(\overline{\bu} ,\overline{\bT} )
$ given after \eqref{solute2} to expand it as 
\begin{align*}
\int_0^t
K_7\dt' 
&
=
\int_0^t
\int_{\Omega_{\zeta}}
 (1-J_{\eta-\zeta})\partial_{t'}\overline{\bT} 
 :
 (\overline{\bT}-\bU)
 \dx\dt' 
-
\int_0^t
\int_{\Omega_{\zeta}} J_{\eta-\zeta} \nabx \overline{\bT} \cdot\partial_{t'}\bm{\Psi}_{\eta-\zeta}^{-1}\circ \bm{\Psi}_{\eta-\zeta} :(\overline{\bT}-\bU) \dx\dt' 
\\&\qquad 
+
\int_0^t
\int_{\Omega_{\zeta}} 
\overline{\bu}  \cdot\nabx \overline{\bT}  (\mathbb{I}-\mathbb{B}_{\eta-\zeta}):(\overline{\bT}-\bU)\dx\dt' 
+
\int_0^t
\int_{\Omega_{\zeta}} 
\mathbb{W}(\nabx\overline{\bu}) (\mathbb{B}_{\eta-\zeta}-\mathbb{I})\overline{\bT} :(\overline{\bT}-\bU)\dx\dt' 
\\&\qquad
+
\int_0^t
\int_{\Omega_{\zeta}} 
\overline{\bT} (\mathbb{B}_{\eta-\zeta}-\mathbb{I})^\top\mathbb{W}( (\nabx\overline{\bu} )^\top):(\overline{\bT}-\bU)\dx\dt' 
\\&
=:\int_0^t(K_7^a+K_7^b+K_7^c+K_7^d+K_7^e)\dt'
\end{align*} 
where
\begin{align*}
\int_0^t
 K_7^a\dt' 
&\lesssim
\int_0^t
\Vert
\overline{\bT}-\bU\Vert_{L^2(\Ozeta)}
\Vert \eta-\zeta \Vert_{W^{1,\infty}(\omega)}\Vert\partial_{t'}\overline{\bT} \Vert_{L^2(\Ozeta)}\dt'
\\
&\leq
\delta_7^a
\sup_{t'\in (0,t]}
\Vert (\overline{\bT}-\bU)(t') \Vert_{L^{2}( \Omega_{\zeta})}^2
+
c(\delta_7^a) 
\int_0^t
\Vert \eta-\zeta \Vert_{W^{2,2}(\omega)}^2
\Vert \partial_{t'}\overline{\bT} \Vert_{L^{2}( \Omega_{\zeta})}^2
\dt'
\end{align*}
hold for any $\delta_7^a>0$ and for any $t\in I$.
On the other hand, since
\begin{align*}
\sup_{t'\in (0,t]}\Vert
(\overline{\bT}-\bU)(t')\Vert_{L^6(\Ozeta)}^2
\lesssim \sup_{t'\in (0,t]}
\big(
\Vert  \overline{\bT}(t')  \Vert_{W^{1,2}( \Omega_{\zeta})}^2
+
\Vert \bU(t') \Vert_{L^{\infty}( \Omega_{\zeta})}^2
\big)
\end{align*}
is finite, it follows that
\begin{align*}
\int_0^t
 K_7^b\dt' 
&\lesssim
\int_0^t
\Vert
\overline{\bT}-\bU\Vert_{L^6(\Ozeta)}
\Vert \eta-\zeta \Vert_{W^{1,6}(\omega)}\Vert \overline{\bT} \Vert_{W^{1,6}(\Ozeta)}
\Vert \partial_{t'}(\eta-\zeta) \Vert_{L^{2}(\omega)}\dt' 
\\
&\leq
\delta_7^b
\sup_{t'\in (0,t]}
\Vert \partial_{t'}(\eta-\zeta)(t') \Vert_{L^{2}(\omega)}^2
+
c(\delta_7^b) 
\int_0^t
\Vert \eta-\zeta \Vert_{W^{2,2}(\omega)}^2
\Vert \overline{\bT} \Vert_{W^{2,2}( \Omega_{\zeta})}^2
\dt'
\end{align*} 
and by using the fact that $\Vert  \overline{\bu}(t)  \Vert_{W^{1,2}( \Omega_{\zeta})}^2$ is essentially bounded in time, we obtain
\begin{align*}
\int_0^t
 K_7^c\dt' 
&\lesssim
\int_0^t
\Vert
\overline{\bT}-\bU\Vert_{L^2(\Ozeta)}
\Vert
\overline{\bu} \Vert_{L^4(\Ozeta)}
\Vert \overline{\bT} \Vert_{W^{1,4}(\Ozeta)}
\Vert \eta-\zeta\Vert_{W^{1,\infty}(\omega)}
\dt' 
\\
&\leq
\delta_7^c
\sup_{t'\in (0,t]}
\Vert  (\overline{\bT}-\bU)(t') \Vert_{L^{2}(\omega)}^2
+
c(\delta_7^c) 
\int_0^t
\Vert \eta-\zeta \Vert_{W^{2,2}(\omega)}^2
\Vert \overline{\bT} \Vert_{W^{2,2}( \Omega_{\zeta})}^2
\dt'.
\end{align*} 
Moreover,
\begin{align*}
\int_0^t
(K_7^d+K_7^e)\dt' 
\leq&2
\int_0^t\int_{\Ozeta}
\vert\overline{\bT} \vert\,\vert\mathbb{B}_{\eta-\zeta}-\mathbb{I}\vert\,\vert\nabx\overline{\bu} \vert\,\vert \overline{\bT}-\bU\vert
\dx\dt'
\\&\lesssim
\int_0^t
\Vert \overline{\bT}\Vert_{L^6(\Ozeta)}
\Vert \eta-\zeta \Vert_{W^{1,6}(\omega)}
\Vert \overline{\bu}\Vert_{W^{1,6}(\Ozeta)}
\Vert \overline{\bT}-\bU\Vert_{L^2(\Ozeta)}
\dt'
\\&
\leq
\delta_7^d
\sup_{t'\in (0,t]}
\Vert (\overline{\bT}-\bU)(t') \Vert_{L^{2}( \Omega_{\zeta})}^2
+
c(\delta_7^d) 
\int_0^t
\Vert \eta-\zeta \Vert_{W^{2,2}(\omega)}^2
\Vert \overline{\bu} \Vert_{W^{2,2}( \Omega_{\zeta})}^2
\dt'
\end{align*}
since $\Vert \overline{\bT}\Vert_{L^6(\Ozeta)}$ is essentially bounded in time. 
\\
If we now collect all the estimates for the $K_i$'s, then for any $\delta\leq \min\{\delta_1, \delta_2,\ldots, \delta_7^c,\delta_7^d\}$, we obtain  from \eqref{diffSolu1} that
\begin{equation}
\begin{aligned} \label{diffSolu2}
\Vert (\overline{\bT}-\bU)(t)\Vert_{L^2(\Omega_{\zeta})}^2 
\leq& 
\Vert \overline{\bT}_0-\bU_0\Vert_{L^2(\Omega_{\zeta_0})}^2
+
\delta
\sup_{t'\in (0,t]}\Vert (\overline{\bT}-\bU)(t')\Vert_{L^2(\Omega_{\zeta})}^2
\\&
c(\delta)
\int_0^t
 \Vert \overline{\bT} -\bU \Vert_{L^2(\Omega_{\zeta})}^2
\Vert  \bU \Vert_{L^{\infty}(\Omega_{\zeta})}^2\dt'  
+
 \delta 
\int_0^t
\Vert \nabx( \overline{\bu} - \bv) \Vert_{L^2(\Omega_{\zeta})}^2\dt'
\\&+
\delta 
\sup_{t'\in (0,t]}
\Vert \partial_{t'}(\eta-\zeta)(t') \Vert_{L^{2}( \omega)}^2 
 +
\delta 
\sup_{t'\in (0,t]}
\Vert (\eta-\zeta)(t') \Vert_{W^{2,2}( \omega)}^2  
\\&+
c(\delta) 
\int_0^t
\big(
 \Vert \overline{\bu} - \bv \Vert_{L^2(\Omega_{\zeta})}^2
 +
 \Vert \eta-\zeta \Vert_{W^{2,2}(\omega)}^2
 \big) 
\Vert  \overline{\bT} \Vert_{W^{2,2}(\Omega_{\zeta})}^2
 \dt'    
\\&
+
c(\delta) 
\int_0^t
\Vert \eta-\zeta \Vert_{W^{2,2}(\omega)}^2
\big( 
\Vert \partial_{t'}\overline{\bT} \Vert_{L^{2}( \Omega_{\zeta})}^2 
+
\Vert  \overline{\bu} \Vert_{W^{2,2}(\Omega_{\zeta})}^2  
\big)
\dt'  
\\
&+ 
\varepsilon^2c(\delta)
\int_0^t
\big(\Vert \overline{\bT} \Vert_{W^{2,2}( \Omega_{\zeta})}^2 
+
\Vert \bU \Vert_{L^{\infty}( \Omega_{\zeta})}^2
\big)
\dt'
\end{aligned}
\end{equation}
holds for any $\delta>0$ and for any $t\in I$.

\subsection{Full estimate}
In order to get our desired estimate  \eqref{contrEst0}, we sum up the estimates \eqref{velShellDiff2}  and \eqref{diffSolu2} to obtain 
\begin{equation}
\begin{aligned} 
\big(&\Vert (\overline{\bu}-\bv)(t)\Vert_{L^2(\Omega_{\zeta})}^2
+
 \Vert  \partial_t(\eta-\zeta)(t) \Vert_{L^{2}(\omega )}^2
+
\Vert  \naby^2(\eta-\zeta)(t) \Vert_{L^{2}(\omega )}^2
+
\Vert (\overline{\bT}-\bU)(t)\Vert_{L^2(\Omega_{\zeta})}^2 
 \big)
\\&\qquad+ 
\int_0^t
\Vert \nabx(\overline{\bu}-\bv)\Vert_{L^2(\Omega_{\zeta})}^2\dt'
+
\gamma 
\int_0^t\Vert  \partial_{t'}\naby(\eta-\zeta) \Vert_{L^{2}(\omega )}^2
\dt'
\\
\lesssim & 
\Vert \overline{\bu}_0-\bv_0\Vert_{L^2(\Omega_{\zeta})}^2
+
 \Vert   \eta_\star-\zeta_\star \Vert_{L^{2}(\omega )}^2
+
\Vert  \naby^2(\eta_0-\zeta_0) \Vert_{L^{2}(\omega )}^2 
+
\Vert \overline{\bT}_0-\bU_0\Vert_{L^2(\Omega_{\zeta_0})}^2
   \\&
+ 
\int_0^t
\Vert \overline{\bu} - \bv\Vert_{L^2(\Ozeta)}^2
 \big(1+\Vert \overline{\bT}\Vert_{W^{2,2}(\Ozeta)}^2 \big)
 \dt'
\\&
+ 
\int_0^t 
  \Vert \partial_{t'}(\eta-\zeta)\Vert_{L^{2}(\omega)}^2
   \Vert \overline{\bu} \Vert_{W^{2,2}(\Ozeta)}^2  
 \dt'
\\& 
+ 
 \int_0^t
\Vert \eta-\zeta\Vert_{W^{2,2}(\omega)}^2
\big( c(\gamma)+\Vert \partial_{t'}\overline{\bu}  \Vert_{L^2(\Omega_{\zeta})}^2
+
\Vert \overline{\bu}  \Vert_{W^{2,2}(\Ozeta)}^2 
\big)\dt' 
\\& 
+  \int_0^t
\Vert \eta-\zeta\Vert_{W^{2,2}(\omega)}^2
\big( \Vert \overline{p}  \Vert_{W^{1,2}(\Ozeta)}^2
+\Vert\partial_{t'} \overline{\bT}  \Vert_{L^{2}(\Ozeta)}^2
+\Vert \overline{\bT}  \Vert_{W^{2,2}(\Ozeta)}^2
\big)\dt' 
\\
&
+ 
\int_0^t
\Vert  \overline{\bT} - \bU\Vert_{L^2(\Ozeta)}^2\big(1+\Vert \bU\Vert_{L^\infty(\Ozeta)}^2 \big)\dt'
\\
&+ 
\varepsilon^2 
\int_0^t
\big(\Vert \overline{\bT} \Vert_{W^{2,2}( \Omega_{\zeta})}^2 
+
\Vert \bU \Vert_{L^{\infty}( \Omega_{\zeta})}^2
\big)
\dt'
\end{aligned}
\end{equation}
Our desired estimate now follows from Gr\"onwall's lemma.

\end{proof}
Having proved Proposition \ref{prop:zeroCorotational}, we now  set $(\eta, \bu, p, \bT)=(\eta^\varepsilon, \bu^\varepsilon, p^\varepsilon,  \bT^\varepsilon)_{\varepsilon>0}$  in  \eqref{divfree}-\eqref{solute}.  
We note that if the corresponding dataset $(  \bT_0, \bu_0, \eta_0, \eta_\star)$ satisfies \eqref{dataConv}, then the transformed dataset $(  \overline{\bT}^\varepsilon_0, \overline{\bu}^\varepsilon_0, \eta^\varepsilon_0, \eta^\varepsilon_\star)$ satisfies
\begin{align*}
&\eta^\varepsilon_0 \rightarrow \zeta_0 \quad\text{in}\quad W^{2,2}( \omega), 
\\&\eta^\varepsilon_\star \rightarrow \zeta_\star \quad\text{in}\quad L^2( \omega),
\\&\overline{\bu}^\varepsilon_0  \rightarrow \bv_0 \quad\text{in}\quad L^2( \Omega_{\zeta(0)}), 
\\&
\overline{\bT}^\varepsilon_0  \rightarrow \bU_0 \quad\text{in}\quad L^2( \Omega_{\zeta(0)})
\end{align*}
as $\varepsilon\rightarrow0$.
Thus,  an immediate corollary of Proposition \ref{prop:zeroCorotational} is the following result.
\begin{corollary}
\label{cor:main}
Let $(\eta^\varepsilon,\overline{\bu}^\varepsilon, \overline{p}^\varepsilon, \overline{\bT}^\varepsilon)_{\varepsilon>0}$ be a collection of strong solutions of \eqref{divfree}-\eqref{interface} with dataset $( \overline{\bT}^\varepsilon_0, \overline{\bu}^\varepsilon_0, \eta^\varepsilon_0, \eta^\varepsilon_\star)$ for which \eqref{dataConv} holds with $ \bU_0$ satisfying \eqref{mainDataForAllBounded}. Then we have
\begin{align*}
&\eta^\varepsilon \rightarrow \zeta \quad\text{in}\quad L^\infty(I;W^{2,2}( \omega)),  
\\&\partial_t\eta^\varepsilon \rightarrow \partial_t\zeta \quad\text{in}\quad L^\infty(I;L^{2}( \omega)),
\\&\partial_t\eta^\varepsilon \rightarrow \partial_t\zeta \quad\text{in}\quad L^2(I;W^{1,2}( \omega)),
\\&\overline{\bu}^\varepsilon  \rightarrow \bv \quad\text{in}\quad L^\infty(I;L^2( \Omega_{\zeta})), 
\\
&\nabx\overline{\bu}^\varepsilon  \rightarrow \nabx\bv \quad\text{in}\quad L^2(I;L^2( \Omega_{\zeta})),
\\&
\overline{\bT}^\varepsilon  \rightarrow \bU \quad\text{in}\quad L^\infty(I;L^2( \Omega_{\zeta}))
\end{align*}
where $(\zeta,\bv,  \bU)$ is an essentially bounded weak solution of \eqref{divfreeL}-\eqref{interfaceL} with dataset $(  \bU_0, \bv_0, \zeta_0, \zeta_\star)$.
\end{corollary}
Now we note that since
\begin{equation}
\begin{aligned}
\label{triIneq}
\Vert \bm{1}_{\Omega_{\eta^\varepsilon}}
\bu^\varepsilon  - \bm{1}_{\Omega_{\zeta}}\bv \Vert_{L^2(\Omega \cup S_\ell)}^2
\lesssim&
\Vert(\bm{1}_{\Omega_{\eta^\varepsilon}}
  - \bm{1}_{\Omega_{\zeta}}) \bu^\varepsilon \Vert_{L^2(\Omega \cup S_\ell)}^2
+
\Vert
\bm{1}_{\Omega_{\zeta}}( \overline{\bu}^\varepsilon  - \bv)
\Vert_{L^2(\Omega \cup S_\ell)}^2
\\&+
  \Vert
\bm{1}_{\Omega_{\zeta}}(\bu^\varepsilon -\overline{\bu}^\varepsilon)
\Vert_{L^2(\Omega \cup S_\ell)}^2,
\end{aligned}
\end{equation}
the first two terms to the right converge to zero as $\varepsilon\rightarrow0$ due to Corollary \ref{cor:main}.
%
%
Also, since $ \overline{\bu}^\varepsilon = \bu^\varepsilon \circ \bm{\Psi}_{\eta^\varepsilon-\zeta}= \bu^\varepsilon (t, \bx+\bn(\eta^\varepsilon-\zeta)\phi)$, the last term to the right in \eqref{triIneq} also converges to zero as $\varepsilon\rightarrow0$ due to the strong uniform convergence of $\eta^\varepsilon$ to $\zeta$. Thus, it follows that
\begin{align*}
\bm{1}_{\Omega_{\eta^\varepsilon}} \bu^\varepsilon  \rightarrow \bm{1}_{\Omega_{\zeta}}\bv \quad\text{in}\quad L^\infty(I;L^2(\Omega \cup S_\ell)).
\end{align*}
Similarly, we obtain strong convergence of $\bm{1}_{\Omega_{\eta^\varepsilon}}(\nabx\bu^\varepsilon, 
\bT^\varepsilon )$ to their respective limits as stated in Theorem \ref{thm:main2} thereby finishing the proof.

\section*{Statements and Declarations} 
\subsection*{Funding}
This work has been partly supported by Grant ME 6391/1-1 (543675748) by the German Research
\subsection*{Author Contribution}
The authors wrote and reviewed the manuscript.
\subsection*{Conflict of Interest}
The authors declares that they have no conflict of interest.
\subsection*{Data Availability Statement}
Data sharing is not applicable to this article as no datasets were generated
or analyzed during the current study.
\subsection*{Competing Interests}
The authors have no competing interests to declare that are relevant to the content of this article.
%


\end{document}